\documentclass[11pt,a4paper, reqno]{amsart}
\usepackage{mathrsfs}

\usepackage{amsmath,amsfonts,verbatim}
\usepackage{latexsym}
\usepackage{amssymb,leftidx}
\usepackage{extarrows}
\usepackage{overpic}
\usepackage{color}
\usepackage{epsfig}
\usepackage{subfigure}
\usepackage{tikz}
\usepackage[backref=page]{hyperref}

\hypersetup{urlcolor=blue, citecolor=red}

\usepackage{amssymb}
\usepackage{mathrsfs}
\usepackage{amscd}
\usepackage{bbm}
\usepackage{cite}

\newcommand\R{\mathbb{R}}

\newcommand\Z{\mathbb{Z}}
\newcommand\N{\mathbb{N}}

\newcommand\A{\bf A}

\newcommand\CC{\mathcal{C}}

\newcommand\LL{\mathcal{L}}

\usepackage{graphicx}
\usepackage{float}

\numberwithin{equation}{section}
\newtheorem{proposition}{Proposition}[section]

\newtheorem{lemma}{Lemma}[section]
\newtheorem{theorem}{Theorem}[section]

\newtheorem{remark}{Remark}[section]

\newcommand{\conR}{\mathsf{R}_{\mathsf{Conj}}}

\newcommand{\CS}{Y}

\begin{document}
\title[Strichartz estimates]{Strichartz estimates for the Schr\"odinger equation \\ in high dimensional critical electromagnetic fields }

\author{Qiuye Jia}
\address{Department of Mathematics, the Australian National University; }
\email{Qiuye.Jia@anu.edu.au; }

\author{Junyong Zhang}
\address{Department of Mathematics, Beijing Institute of Technology, Beijing 100081; }
\email{zhang\_junyong@bit.edu.cn; }

\begin{abstract}
We prove Strichartz estimates for the Schr\"odinger equation with scaling-critical electromagnetic potentials in dimensions $n\geq3$.
The decay assumption on the magnetic potentials is critical, including the case of the Coulomb potential. 
Our approach introduces novel techniques, notably the construction of Schwartz kernels for the localized Schr\"odinger propagator, which separates the antipodal points of $\mathbb{S}^{n-1}$,
 in these scaling critical electromagnetic fields. This method enables us to prove the $L^1(\R^n)\to L^\infty(\R^n)$ for the localized Schrödinger propagator, as well as global Strichartz estimates. Our results provide a positive answer to the open problem posed in \cite{DFVV, Fanelli, FV}, and fill a longstanding gap left by \cite{EGS1, EGS2}.
\end{abstract}

\maketitle

\begin{center}
 \begin{minipage}{120mm}
   { \small {\bf Key Words:  Decay estimates,  Strichartz estimates,  Singular electromagnetic potentials, Schr\"odinger equation}
      {}
   }\\
    { \small {\bf AMS Classification:}
      { 42B37, 35Q40, 35Q41.}
      }
 \end{minipage}
 \end{center}


\tableofcontents

\section{Introduction and main results}

\subsection{The setting and motivation}

We consider the Hamiltonian of a nonrelativistic charged particle in an electromagnetic field, given by the operator
\begin{equation}\label{H-AV}
H_{A, V}=(i\nabla+A(x))^2+V(x),
\end{equation}
where the electric scalar potential $V: \R^n\to \R$ and the magnetic vector potential
\begin{equation}
A(x)=(A^1(x),\ldots, A^n(x)): \, \R^n\to \R^n.
\end{equation}
The magnetic potential $A(x)$ characterizes the interaction of a free particle with an external magnetic field.
The Schr\"odinger operators with electromagnetic potentials have been extensively studied from the perspectives of spectral and scattering theory. 
Various important physical potentials, such as the constant magnetic field and the Coulomb potential, were explored by Avron, Herbst, and Simon \cite{AHS1, AHS2, AHS3} and Reed and Simon \cite{RS}.
This paper, building on recent studies \cite{DFVV, EGS1, EGS2, FFFP, FFFP1, GYZZ, FV, FZZ}, aims to investigate the role of electric and magnetic potentials in the short- and long-time behavior of solutions to dispersive equations, including the Schrödinger, wave, and Klein-Gordon equations. \vspace{0.2cm}

The study of decay estimates and Strichartz estimates for dispersive equations has a long history, due to their central importance in both analysis and the theory of partial differential equations (PDEs).
We refer to \cite{BPSS, BPST, BG, CS,DF, DFVV, EGS1, EGS2, S} and the references therein for those estimates for the classical and important Schr\"odinger and wave equations with electromagnetic potentials in mathematical and physical fields. However, since different potentials have different effects, it is hard to provide a universal argument to treat potentials of different types. Consequently, the picture of the program is far from being complete, particularly in the case of critical physical potentials. In this direction, there is a substantial body of literature studying the decay behavior of dispersive equations under perturbations by various potentials. Even for subcritical magnetic potentials, several papers (see \cite{CS, DF, DFVV, EGS1, EGS2, S} and the references therein) have addressed time-decay and Strichartz estimates. For the scaling-critical purely inverse-square electric potential, pioneering results were established by Burq, Planchon, Stalker, and Tahvildar-Zadeh \cite{BPSS, BPST}, in which they established Strichartz estimates for the Schr\"odinger and wave equations, in space dimension $n\geq2$. When a magnetic field is present, the situation becomes more complicated. This is because the scaling-critical magnetic potential induces a long-range perturbation. To focus our discussion, we consider the scaling-critical electromagnetic Schrödinger operator
 \begin{equation}\label{LAa}
\mathcal{L}_{{\A},a}=\Big(i\nabla+\frac{{\A}(\hat{x})}{|x|}\Big)^2+\frac{a(\hat{x})}{|x|^2},\qquad x\in \R^n\setminus\{0\}
\end{equation}
where $\hat{x}=\tfrac{x}{|x|}\in\mathbb{S}^{n-1}$, $a\in W^{1,\infty}(\mathbb{S}^{n-1},\mathbb{R})$ and ${\A}\in W^{1,\infty}(\mathbb{S}^{n-1};\R^n)$  satisfies the transversal gauge condition (known as the Cr\"onstrom gauge)
\begin{equation}\label{eq:transversal}
{\A}(\hat{x})\cdot\hat{x}=0,
\qquad
\text{for all }\hat{x}\in\mathbb{S}^{n-1}.
\end{equation}
In \cite{FFFP, FFFP1}, Fanelli, Felli, Fontelos, and Primo studied the time-decay estimate for the Schr\"odinger equation associated with the operator \eqref{LAa} when $n=2$ and $n=3$. 
More precisely,  they \cite{FFFP1} mainly considered 2D the {\it Aharonov-Bohm} potential
\begin{equation}\label{ab-potential}
a\equiv0,
\qquad
{\A}(\hat{x})=\alpha\Big(-\frac{x_2}{|x|},\frac{x_1}{|x|}\Big),\quad \alpha\in\R,
\end{equation}
and the 3D {\it inverse-square} potential
\begin{equation}\label{eq:inversesquare}
{\A}\equiv0,
\qquad
a(\hat{x})\equiv a>-1/4.
\end{equation}
In \cite{FFFP}, they focused on the two dimension model and proved the time-decay estimate for the Schr\"odinger equation
\begin{equation}\label{eq:decayshro}
\|e^{it\LL_{{\A},a}}\|_{L^1(\R^2)\to L^\infty(\R^2)}\lesssim |t|^{-1}
\end{equation}
 provided that
\begin{equation}\label{equ:condassa}
  \|a_-\|_{L^\infty(\mathbb{S}^1)}<\min_{k\in\Z}\{|k-\Phi_{\A}|\}^2,
  \qquad
  \Phi_{\A}\notin\Z,
\end{equation}
where $a_-:=\max\{0,-a\}$ is the negative part of $a$, and $\Phi_{\A}$ is the total flux along the sphere
\begin{equation}\label{equ:defphia1}
  \Phi_{\A}=\frac{1}{2\pi}\int_0^{2\pi} \alpha(\theta)\;d\theta,
\end{equation}
with $\alpha(\theta)$ defined by 
\begin{equation}\label{equ:alpha}
\alpha(\theta)={\bf A}(\cos\theta,\sin\theta)\cdot (-\sin\theta,\cos\theta).
\end{equation}
So the Strichartz estimates for $e^{it\LL_{{\A},a}}$ are  consequences of \eqref{eq:decayshro} and the standard Keel-Tao argument \cite{KT}. It is well known that the magnetic potential $A\sim {\A}/|x|\sim 1/|x|$ is critical for the validity of Strichartz estimates, as demonstrated, for example, in \cite{FG} for the Schr\"odinger equation.
However, the approach in \cite{FFFP, FFFP1} does not extend to the wave and Klein-Gordon equations due to the absence of  pseudoconformal invariance, a property used in the Schr\"odinger equation. More recently, Fanelli, Zheng and the last author \cite{FZZ}
proved the Strichartz estimate for wave equations by constructing the propagator $\sin(t\sqrt{\mathcal L_{{\A},0}})/\sqrt{\mathcal L_{{\A},0}}$ based on \emph{Lipschitz-Hankel integral formula} and establishing the local smoothing estimates. Additionally, Gao, Yin, Zheng and the last author \cite{GYZZ} constructed the spectral measure and further proved time-decay and Strichartz estimates for the Klein-Gordon equation. \vspace{0.2cm}

We remark that aforementioned results for the scaling critical magnetic case are currently only valid in two dimensions. 
The success of the argument relies on the simple structures of the cross section $\mathbb{S}^1$ and the potentials, in which there are no conjugate points, and the eigenfunctions and eigenvalues of the magnetic 
Laplacian on $\mathbb{S}^1$ are explicit. 
Dispersive estimates with a scaling-critical class of electric potentials (with respect to the global Kato norm) is proven in \cite{BG}. 
However, to our best knowledge, no Strichartz estimate has been established for scaling-critical magnetic Schr\"odinger operators in higher dimensions $n\geq3$ with the presence of both the singular magnetic and electric potentials. 
For higher dimensions, we refer to \cite{DFVV, BK, FV, EGS1, EGS2} for results on almost-critical magnetic Strichartz estimates, though the critical Coulomb case remains unaddressed. This Schr\"odinger operator \eqref{LAa}, which includes the critical Coulomb-type decay potential, has attracted significant interest from both the mathematical and physical communities. 
For example, the Aharonov-Bohm effect \cite{ES1949, AB59} arises from the critical Coulomb type decay potential, and the diffractive behavior of the wave in such potentials has been studied in \cite{Yang1,Yang2}. The asymptotic behavior of the time-independent  Schr\"odinger 
solution was analyzed in \cite{FFT}. Motivated by these developments and to address open problems left in \cite{DFVV, EGS1, EGS2, FV}, we aim to prove Strichartz estimates for Schr\"odinger equations associated with the electromagnetic Schr\"odinger operator \eqref{LAa} in dimension $n\geq3$,  where both the electric and magnetic potentials are singular at origin and scaling critical. 
\vspace{0.2cm}

It has been known that a nontrapping assumption on the magnetic field, generated by the magnetic potential 
$A$, is necessary to ensure long-time dispersion. We refer the reader to \cite{FV, DFVV} for more details. In three dimensions, the magnetic vector potential $A$ produces the magnetic field $B$ given by
\begin{equation}\label{B-3}
B(x)=\mathrm{curl} (A)=\nabla\times A(x).
\end{equation}
In general dimension $n\geq2$, $B$ should be regarded as matrix-valued field $B:\R^n\to \mathcal{M}_
{n\times n}(\R)$ given by
\begin{equation}\label{B-n}
B:=DA-DA^t,\quad B_{ij}=\frac{\partial A^i}{\partial x_j}-\frac{\partial A^j}{\partial x_i}.
\end{equation}
In particular, the tangential part (i.e. trapping component) of $B$ is defined by 
\begin{equation}\label{B-tau}
B_{\tau}(x)=\frac{x}{|x|}B(x),\quad \big(B_{\tau}(x)\big)_j=\sum_{k=1}^n\frac{x_k}{|x|} B_{kj}.
\end{equation}
In particular $n=3$, $B_{\tau}(x)$ is the projection of $B=\mathrm{curl} (A)$ on the tangential space in $x$ to the sphere of radius $|x|$.
If $B_{\tau}(x)\cdot x=0$ for any $n\geq2$, hence $B_{\tau}(x)$ is a tangential vector field in any dimension. 
The trapping component may be interpreted as an obstruction to the dispersion of solutions. As stated in \cite{FV}, some explicit examples of potentials $A$ with $$B_{\tau}(x)=\frac{x}{|x|}\wedge \mathrm{curl} (A)=0$$
in dimension three.
For example, in $\R^3$,
\begin{equation}
A=\frac{(-x_2,x_1,0)}{|x_1|^2+|x_2|^2+|x_3|^2}=\frac{(x_1,x_2,x_3)}{|x_1|^2+|x_2|^2+|x_3|^2}\wedge(0, 0,1),
\end{equation}
then one can check that
\begin{equation}
\nabla\cdot A=0,\quad B=-2\frac{x_3}{\big(|x_1|^2+|x_2|^2+|x_3|^2\big)^2} (x_1,x_2,x_3),\quad B_\tau=0.
\end{equation}
Another 2D type example is the aforementioned Aharonov-Bohm potential studied in \cite{FFFP, FFFP1,GYZZ, FZZ}
\begin{equation}
A=\frac{(-x_2,x_1)}{|x_1|^2+|x_2|^2},
\end{equation}
which satisfies $B_{\tau}(x)=0$. In this paper, we focus on $\LL_{{\A},a}$, where $A(x)={\A}(\hat{x})/|x|$ satisfies \eqref{eq:transversal}. These two conditions imply that $B_{\tau}(x)=0$,
making it natural to study the long-time dispersion behavior of the dispersive equations associated with $\LL_{{\A},a}$.

\subsection{Main results}  In the flat Euclidean space, the free Schr\"odinger equation reads
 \begin{equation}\label{equ:E-S}
\begin{cases}
i\partial_{t}u-\Delta u=0, \quad (t,x)\in I\times\R^n; \\ u(0)=f(x).
\end{cases}
\end{equation}
It is well known by \cite{GV, KT} that there exists a constant $C>0$ such that
\begin{equation*}
\begin{split}
&\|u(t,x)\|_{L^q_t(I;L^p_x(\R^n))}\leq C \|f\|_{ \dot H^s(\R^n)},
\end{split}
\end{equation*}
where $I$ is a subset interval of $\R$ and $\dot{H}^s(\R^n)$ is the usual homogeneous Sobolev space, and the pair
$(q, p)$ is an \emph{admissible pair} at $\dot H^s$-level, that is, for $s\geq0$
\begin{equation}\label{adm-p}
(q,p)\in\Lambda_s:=\big\{2\leq q,p\leq\infty, \quad 2/q+n/p=n/2-s,\quad (q,p,n)\neq(2,\infty,2)\big\}.
\end{equation}
In particular, when $I=\R$, we say that the Strichartz estimates are global-in-time.  \vspace{0.2cm}

Throughout this paper, pairs of conjugate indices will be written as $p, p'$, meaning that $\frac{1}p+\frac1{p'}=1$ with $1\leq p\leq\infty$. \vspace{0.2cm}

\noindent

We now state the first main result.

\begin{theorem}[Strichartz estimate]\label{thm:Stri} Let  $\LL_{{\A},a}$ be the operator defined in \eqref{LAa} on $\R^n$ with $n\geq3$, where $a\in C^{\infty}(\mathbb{S}^{n-1},\mathbb{R})$  and ${\A}\in C^{\infty}(\mathbb{S}^{n-1},\mathbb{R}^n)$ satisfies \eqref{eq:transversal}.
Assume $P_{{\A}, a}:=(i\nabla_{\mathbb{S}^{n-1}}+{\A}(\hat{x}))^2+a(\hat{x})+(n-2)^2/4$ is a strictly positive operator on $L^2(\mathbb{S}^{n-1})$. Then the homogenous Strichartz estimates
\begin{equation}\label{Str-est}
\|e^{it\LL_{{\A},a}}u_0\|_{L^q_tL^p_x(\mathbb{R}\times \R^n)}\leq
C\|u_0\|_{L^2(\R^n)}
\end{equation}
hold for admissible pairs $(q,p)\in\Lambda_0$ that satisfy
\eqref{adm-p}. Moreover, the inhomogeneous Strichartz estimates
\begin{equation}\label{eq:inhom}
\Big\|\int_0^t e^{i(t-s)\LL_{{\A},a}}F(s) ds\Big\|_{L^q_tL^p_x(\mathbb{R}\times \R^n)}\leq C \| F
\|_{L^{\tilde q'}_tL^{\tilde{p}'}_x(\mathbb{R}\times \R^n)}
\end{equation}
hold for admissible pairs $(q, p)$, $(\tilde q, \tilde{p})\in \Lambda_0$, except for the double endpoint case $q=\tilde{q}=2$.
\end{theorem}

\noindent

\begin{remark}\label{rem:end-point} The homogeneous Strichartz estimate is sharp and includes the endpoint $(q,p)=(2, \frac{2n}{n-2})$. However, the inhomogeneous Strichartz estimate at the double endpoint $q=\tilde{q}=2$ is not proven, due to technical limitations of our localized method. In addition, we require $a,\A$ to be smooth since we are stating results for general dimensions and do not persue sharpness on this for each dimension. As pointed out by Schlag \cite{SchlagSurvey}, Goldberg and Visan \cite{GoldbergVisan} showed that $C^{\frac{n-3}{2}}$-regularity is necessary in odd dimensions, and Erdo\u gan and Green \cite{ErdoganGreen} showed that this indeed is sufficient in dimension 5 and 7.
\end{remark}

\textcolor{blue}{Now let us figure out some key points in our proof:}

\begin{itemize}

\item We replace the usual perturbation argument by constructing the kernel of propagator. The Coulomb-type potential considered here is not included in  \cite{DFVV, EGS1, EGS2}, since the perturbation arguments in those works break down for this scaling critical potential. More precisely, the coulomb potential does not satisfy the assumptions \cite[(1.9)]{DFVV}, \cite[(4)]{EGS2}. To treat $\mathcal L_{{\A},a}$ as a perturbation of $-\Delta$, as done in those papers, one would consider 
$$\mathcal{L}_{{\A},a}=-\Delta+\frac{|{\A}(\hat{x})|^2+i\,\mathrm{div}_{\mathbb{S}^{n-1}}{\A}(\hat{x})+a(\hat{x})}{|x|^2}+2i\frac{{\A}(\hat{x})}{|x|}\cdot\nabla.$$
However, this approach encounters a the long-range perturbation due to the term $2i\frac{{\A}(\hat{x})}{|x|}\cdot\nabla$, which complicates the treatment.
Specifically, one would require local smoothing estimates to gain an additional derivative, but this seems to be unfeasible. 
Even for the almost critical magnetic potential, the endpoint homogeneous Strichartz estimates on $\R^3$ are still missing, see \cite[Theorem 1.1]{DFVV}, \cite[Theorem 1]{EGS2}.

 In the spirit of \cite{GYZZ, FZZ}, we analyze the Schr\"odinger propagator $e^{it\mathcal{L}_{{\A},a}}$ directly, rather than relying on the perturbation arguments.
Unlike the 2D model discussed above, we have to introduce new ingredients, such as a parametrix, to address the challenges posed by the conjugate points on the unit sphere $\mathbb{S}^{n-1}$ and the lack of explicit eigenfunctions and eigenvalues on $\mathbb{S}^{n-1}$. \textcolor{blue}{In the construction of parametrix, the feature of section cross $\mathbb{S}^{n-1}$ (e.g. the fact that its injective radius is larger than $\frac{\pi}{2}$, hence geodesic loops have to be longer than $\pi$) plays an important role.} \vspace{0.2cm}

\item The effects of the scaling critical electromagnetic potential are non-trivial. For instance, 
the diffraction occurs in this electromagnetic field, and the smallest eigenvalue of the operator $P_{{\A}, a}$ plays a  role in determining the range of admissible pairs $(q,p)$.
In particular, the effect of the smallest eigenvalue on the admissible range of $(q,p)$ can be explicitly seen from Theorem \ref{thm:Strichartz'}, which addresses the Strichartz estimates at the $\dot H^s$-level. 

\end{itemize} 

The second main result concerns the Strichartz estimates at the $\dot H^s$-level.  To present our results, we first introduce some preliminary notations. In the following, we will denote the Sobolev spaces by
\begin{align}\label{def:sobolev}
& \dot H^{s}_{{\A},a}(\R^n):=\mathcal L_{{\A},a}^{-\frac s2}L^2(\R^n),
\qquad
H^s_{{\A},a}(\R^n) :=L^2(\R^n)\cap\dot H^{s}_{{\A},a}(\R^n).
\nonumber
\end{align}
Equivalently, the homogeneous Sobolev norm of $\|\cdot\|_{\dot H^{s}_{{\A},a}(\R^n)}$ can be defined by
\begin{equation}\label{Sobolev-n}
\|f\|_{ \dot H^{s}_{{\A},a}(\R^n)}=\Big(\sum_{j\in\Z}2^{2js}\|\varphi_j(\sqrt{\mathcal L_{{\A},a}})f\|_{L^2(\R^n)}^2\Big)^{1/2}.
\end{equation}
where $s\in\R$ and $\varphi_j(\sqrt{\mathcal L_{{\A},a}})$ is the Littlewood-Paley operator, see Section \ref{sec:LP} for details.
For $n\geq3$ and $-1\leq s\leq 1$, we have  
\begin{equation}\label{nor-eq}
 \dot H^{s}_{{\A},a}(\R^n)\sim \dot H^s(\R^n)
 \end{equation}
by \cite[cf. Lemma 2.3]{FFT} in combination with duality and interpolation.

\vspace{0.2cm}

\begin{theorem}[The Strichartz estimates at $\dot H^s$-level] \label{thm:Strichartz'} Let  $\LL_{{\A},a}$
be the same as in Theorem \ref{thm:Stri}. Let
$\nu_0$ denote the positive square root of the smallest eigenvalue of the operator $P_{{\A}, a}:=(i\nabla_{\mathbb{S}^{n-1}}+{\A}(\hat{x}))^2+a(\hat{x})+(n-2)^2/4$ acting on $L^2(\mathbb{S}^{n-1})$. 
Define
\begin{equation}\label{def:alpha}
\alpha=-(n-2)/2+ \nu_0.
\end{equation}
and
\begin{equation}\label{def:q-alpha}
p(\alpha)=
\begin{cases}
\infty,\quad \alpha\geq 0;\\
\frac{n}{|\alpha|}, \quad -(n-2)/2<\alpha< 0.
\end{cases}
\end{equation}
Then one has the Strichartz estimates
\begin{equation}\label{Str-est-s}
\|e^{it\LL_{{\A},a}}u_0\|_{L^q_tL^p_x(\mathbb{R}\times \R^n)}\leq
C\|u_0\|_{ \dot H^{s}_{{\A},a}(\R^n)},
\end{equation}
where $s\geq0$ and
\begin{equation}\label{adm-p-s'}
(q,p)\in \Lambda_{s,\nu_0}:=\Lambda_s\cap \{(q,p): p<p(\alpha)\},
\end{equation}
where $\Lambda_s$ is given by \eqref{adm-p}. 
The restriction $p<p(\alpha)$ is necessary in the sense that the Strichartz estimates \eqref{Str-est-s} may fail even if $(q,p)\in \Lambda_s$, but  $p\geq p(\alpha)$.
\end{theorem}

\begin{remark}\label{rem:set}
The set $\Lambda_{s,\nu_0}$ is non-empty only if $s\in [0,1+\nu_0)$. Moreover,
$\Lambda_{s,\nu_0}=\Lambda_s$ when $s\in [0,1/2+\nu_0)$, in which case the condition $p<p(\alpha)$
automatically disappears, and while $\Lambda_{s,\nu_0} \subsetneq \Lambda_s$ when $s \in [1/2+\nu_0, 1+\nu_0)$.
\end{remark}

\begin{remark}\label{rem:sobolev}
Due to \eqref{nor-eq}, the homogeneous Sobolev norm of $\|\cdot\|_{\dot H^{s}_{{\A},a}(\R^n)}$ of  \eqref{Str-est-s} can be replaced by the standard Sobolev norm $\|\cdot\|_{\dot H^{s}(\R^n)}$ when $0\leq s\leq 1$.
\end{remark}

\subsection{Links to \texorpdfstring{\cite{JZ}}{previous} and future works} 
In this paper, we prove the Strichartz estimates for the high-dimensional scaling-critical magnetic Schr\"odinger equation, thereby filling the gap in the results of \cite{FFFP, DFVV, EGS1, EGS2}. 
The central idea is to construct the kernel of propagator directly, rather than treating the electromagnetic potentials as a perturbation.
More specifically, the new contribution is the construction of the Schrödinger propagator by microlocally constructing the parametrices of the even wave propagator  $\cos(s\sqrt{P})$ and the Poisson wave propagator $e^{(-s+i\pi)\sqrt{P}}$, where $P=P_{{\A}, a}$ is the operator on the unit sphere $\mathbb{S}^{n-1}$. This approach is inspired by our recent paper \cite{JZ}, where we constructed the global Schr\"odinger propagator on metric cone $X=C(Y)$ and proved the pointwise dispersive estimates, assuming that the conjugate radius (minimal distance between conjugate point pairs) of the section cross $Y$ is strictly greater than $\pi$. However, the model operator \eqref{LAa} corresponds to $Y=\mathbb{S}^{n-1}$, where the conjugate radius is exactly $\pi$, making the global construction in \cite{JZ} inapplicable to the current situation.  Instead of the global parametrix construction, we develop localized parametrix constructions for both the even wave propagator $\cos(s\sqrt{P})$ and the Poisson wave propagator $e^{(-s+i\pi)\sqrt{P}}$. \textcolor{blue}{In contrast to a general cross section 
$Y$ considered in \cite{JZ}, an advantage of $\mathbb{S}^{n-1}$ is that its injectivity radius is $\pi$, which simplifies the argument.} Using the oscillatory integral expression for these parametrices, we obtain a representation of the localized Schr\"odinger propagator. Finally, we prove the localized (rather than global) dispersive estimates and then use a variant of Keel-Tao's abstract argument \cite{KT} to prove the global Strichartz estimates.
The Strichartz estimates at $\dot H^s$-level are derived by establishing the Sobolev embedding. The necessity of the additional requirement on the admissible pairs is proved by constructing a counterexample. 

As mentioned earlier, though the homogeneous Strichartz estimates is proven in the full range, the global pointwise dispersive estimate is not treated and the range for the inhomogeneous Strichartz estimates does not include the double endpoint $q=\tilde{q}=2$.
However, we hope to address these issues in our future works by constructing the Schwartz kernels of the resolvent and spectral measure for this electromagnetic Schr\"odinger operator, both on manifolds and with potentials. \vspace{0.1cm}

For instance, the Schwartz kernels of the resolvent and spectral measure associated with Schrödinger operators on conical singular spaces have been systematically studied by Hassell and Vasy \cite{HV1,HV2} and Guillarmou, Hassell and Sikora \cite{GHS1, GHS2}. These kernels were then used to study resolvent estimates in Guillarmou and Hassell \cite{GH} and the Strichartz estimates in Hassell and the last author \cite{HZ}. 
In contrast to the Laplacian in conical singular spaces, the Schr\"odinger operator \eqref{LAa} in this paper is perturbed by electromagnetic potentials. While it exhibits a similar conical singularity, it is defined on the flat Euclidean space.
Therefore, in this paper, instead of constructing the kernel of the resolvent in the more general geometric setting through the microlocal approach, we focus on the sharp homogeneous Strichartz estimate and provide a self-contained argument suitable for non-microlocal readers.
\vspace{0.1cm}

\subsection{Structure of the paper}
We give a characterization of the Schwartz kernel of the Schr\"odinger propagator through the functional calculus in Section \ref{sec: propagator}. Then a localized parametrix construction for the half-wave propagator on $\mathbb{S}^{n-1}$ is given in Section \ref{sec: parametrix}. In Section \ref{sec:LP}, we discuss the Littlewood-Laley theory associated to the electromagnetic Schr\"odinger operator $\mathcal{L}_{{\A},a}$. Then a localized dispersive estimate is proven in Section \ref{sec:dispersive}. Finally, in Section \ref{sec: strichartz, H0 level} and Section \ref{sec: strichartz, Hs level} we prove Theorem \ref{thm:Stri} and Theorem \ref{thm:Strichartz'} respectively.\vspace{0.2cm}

{\bf Acknowledgments:}\quad  The authors would like to thank Andrew Hassell and Luca Fanelli for their helpful discussions and encouragement. J. Zhang is grateful for the hospitality of the Australian National University and Basque Center for Applied Mathematics when he was visiting Andrew Hassell at ANU and Luca Fanelli at BCAM.
J. Zhang was supported by National key R\&D program of China: 2022YFA1005700, National Natural Science Foundation of China(12171031) and Beijing Natural Science Foundation(1242011);
Q. Jia was supported by the Australian Research Council through grant FL220100072.
\vspace{0.2cm}

\section{Construction of the Schr\"odinger propagator }
\label{sec: propagator}

In this section, we primarily focus on constructing the Schwartz kernel of the Schr\"odinger propagator $e^{it\mathcal L_{{\A},a}}$, as stated in Proposition \ref{prop:Sch-pro}.
The strategy combines the spectral methods from \cite{CT1,CT2} with techniques from \cite{FFT}.

\subsection{Functional calculus} In this subsection, inspired by Cheeger-Taylor \cite{CT1,CT2}, we recall the functional calculus associated with the operator $\mathcal{L}_{{\A},a}$, see also \cite{FZZ} for the two dimensional case and \cite{FFT} for the higher dimensional case.

From \eqref{LAa} and \eqref{eq:transversal}, we write
\begin{equation}\label{LAa-r}
\begin{split}
\mathcal{L}_{{\A},a}&=-\Delta+\frac{|{\A}(\hat{x})|^2+i\,\mathrm{div}_{\mathbb{S}^{n-1}}{\A}(\hat{x})+a(\hat{x})}{|x|^2}+2i\frac{{\A}(\hat{x})}{|x|}\cdot\nabla\\
&=-\partial_r^2-\frac{n-1}r\partial_r+\frac{L_{{\A},a}}{r^2},
\end{split}
\end{equation}
where $\hat{x}\in \mathbb{S}^{n-1}$ and
\begin{equation}\label{L-angle}
\begin{split}
L_{{\A},a}&=(i\nabla_{\mathbb{S}^{n-1}}+{\A}(\hat{x}))^2+a(\hat{x})
\\&=-\Delta_{\mathbb{S}^{n-1}}+\big(|{\A}(\hat{x})|^2+a(\hat{x})
+i\mathrm{div}_{\mathbb{S}^{n-1}} {\A}(\hat{x}) \big)+2i {\A}(\hat{x})\cdot\nabla_{\mathbb{S}^{n-1}}.
\end{split}
\end{equation}

From the classical spectral theory, the spectrum of $L_{{\A},a}$ is formed by a countable family of real eigenvalues with finite multiplicity $\{\mu_k({\A},a)\}_{k=1}^\infty$ enumerated such that
\begin{equation}\label{eig-Aa}
\mu_1({\A},a)\leq \mu_2({\A},a)\leq \cdots
\end{equation}
where we repeat each eigenvalue as many times as its multiplicity, and $\lim\limits_{k\to\infty}\mu_k({\A},a)=+\infty$, see \cite[Lemma A.5]{FFT}.
For each $k\in\N, k\geq1$, let $\psi_k(\hat{x})\in L^2(\mathbb{S}^{n-1})$ be the normalized eigenfunction of the operator $L_{{\A},a}$ corresponding to the $k$-th eigenvalue $\mu_k({\A},a)$, i.e. satisfying that
\begin{equation}\label{equ:eig-Aa}
\begin{cases}
L_{{\A},a}\psi_k(\hat{x})=\mu_k({\A},a)\psi_k(\hat{x}) \quad \hat{x} \in  \mathbb{S}^{n-1},\\
\int_{\mathbb{S}^{n-1}}|\psi_k(\hat{x})|^2 d\hat{x}=1.
\end{cases}\end{equation}
Notice that the operator $P_{{\A},a}$ defined in Theorem \ref{thm:Stri} is related to $L_{{\A},a}$ by
\begin{equation} \label{eq: P-definition}
P=P_{{\A},a}=L_{{\A},a}+(n-2)^2/4,
\end{equation}
thus they have the same eigenfunctions and the difference of their eigenvalues is the constant $(n-2)^2/4$.
Compared with the two dimensional problems considered in \cite{FZZ, GYZZ}, we do not know the explicit formulas of our eigenfunctions and eigenvalues, which makes the problems becomes harder
in high dimensions.
\vspace{0.2cm}

We have the orthogonal decomposition $$L^2(\mathbb{S}^{n-1})=\bigoplus_{k\in\N}h_{k}(\mathbb{S}^{n-1}),$$
where 
\begin{equation}\label{hk}
h_{k}(\mathbb{S}^{n-1})=\text{span}\{\psi_k(\hat{x})\}.
\end{equation}
For $f\in L^2(\R^n)$, we can write $f$ in the following form by separation of variables:
\begin{equation}\label{sep.v}
f(x)=\sum_{k\in\N} c_{k}(r)\psi_k(\hat{x}),
\end{equation}
where
\begin{equation*}
 c_{k}(r)=\int_{\mathbb{S}^{n-1}}f(r,\hat{x})
\overline{\psi_k(\hat{x})}\, d\hat{x}.
\end{equation*}
Hence, on each space $\mathcal{H}^{k}=\text{span}\{\psi_k\}$, from \eqref{LAa-r}, we have
\begin{equation*}
\begin{split}
\LL_{{\A},a}=-\partial_r^2-\frac{n-1}r\partial_r+\frac{\mu_k}{r^2}.
\end{split}
\end{equation*}
Let $\nu=\nu_k=\sqrt{\mu_k+(n-2)^2/4}$,  for $f\in L^2(\R^n)$, we define the Hankel transform of order $\nu$ by
\begin{equation}\label{hankel}
(\mathcal{H}_{\nu}f)(\rho,\hat{x})=\int_0^\infty (r\rho)^{-\frac{n-2}2}J_{\nu}(r\rho)f(r,\hat{x}) \,r^{n-1}dr,
\end{equation}
where the Bessel function of order $\nu$ is given by
\begin{equation}\label{Bessel}
J_{\nu}(r)=\frac{(r/2)^{\nu}}{\Gamma\left(\nu+\frac12\right)\Gamma(1/2)}\int_{-1}^{1}e^{isr}(1-s^2)^{(2\nu-1)/2} ds, \quad \nu>-1/2, r>0.
\end{equation}

Using the functional calculus, for a Borel measurable function $F$ (see \cite{Taylor}), we define $F(\mathcal{L}_{{\A},a})$ by
\begin{equation}\label{funct}
F(\mathcal{L}_{{\A},a}) f(r_1,\hat{x})=\int_0^\infty \int_{\mathbb{S}^{n-1}} K(r_1,\hat{x},r_2,\hat{y}) f(r_2,\hat{y})\; r^{n-1}_2\;dr_2\;d\hat{y}, 
\end{equation}
where
$$K(r_1,\hat{x},r_2,\hat{y})=\sum_{k\in\N}\psi_{k}(\hat{x})\overline{\psi_{k}(\hat{y})}K_{\nu_k}(r_1,r_2),$$
and
\begin{equation}\label{equ:knukdef}
  K_{\nu_k}(r_1,r_2)=(r_1r_2)^{-\frac{n-2}2}\int_0^\infty F(\rho^2) J_{\nu_k}(r_1\rho)J_{\nu_k}(r_2\rho) \,\rho d\rho.
\end{equation}

\subsection{Abstract Schr\"odinger propagator}
In this subsection, we construct the propagator of Schr\"odinger equation,
which is similar to \cite[Proposition 2.1]{JZ}.

\begin{proposition}[Schr\"odinger kernel]\label{prop:Sch-pro} Let  $\LL_{{\A},a}$ be the Schr\"odinger operator given in \eqref{LAa} and let $x=(r_1, \hat{x})\in \R^n$ and $y=(r_2, \hat{y})\in \R^n$. 
Then the kernel of Schr\"odinger propagator can be written as
\begin{equation}\label{S-kernel} 
\begin{split}
e^{-it\LL_{{\A},a}}(x,y)&=e^{-it\LL_{{\A},a}}(r_1, \hat{x}, r_2, \hat{y})\\
 &=\big(r_1 r_2\big)^{-\frac{n-2}2}\frac{e^{-\frac{r_1^2+r_2^2}{4it}}}{2it}
  \Big(\frac1{\pi}\int_0^\pi e^{\frac{r_1r_2}{2it} \cos(s)} \cos(s\sqrt{P})(\hat{x}, \hat{y}) ds\\
  &-\frac{\sin(\pi\sqrt{P})}{\pi}\int_0^\infty e^{-\frac{r_1r_2}{2it} \cosh s} e^{-s\sqrt{P}}(\hat{x}, \hat{y}) ds\Big),
\end{split}
\end{equation}
where $P=P_{{\A},a}=L_{{\A},a}+(n-2)^2/4$ with $L_{{\A},a}$ in \eqref{L-angle}.

\end{proposition}

\begin{remark} In contrast to the 2D model studied in \cite{FFFP, FFFP1,GYZZ, FZZ}, we only obtain an abstract representation of the Schr\"odinger propagator, due to the lack of explicit eigenfunctions and eigenvalues of $P$.
\end{remark}

\begin{remark} In the spirit of our recent paper \cite{JZ}, if we separate the conjugate points on the unit sphere, we can obtain the modified Hadamard parametrix of even wave propagator $\cos(s\sqrt{P})$
 and the Poisson wave propagator $e^{(-s+i\pi)\sqrt{P}}$ on $\mathbb{S}^{n-1}$ (see Section \ref{sec: parametrix} below ).  This allows us to prove decay estimates for a localized  Schr\"odinger propagator.
\end{remark}

\begin{proof}
 From \eqref{funct}, we take $F(\rho^2)=e^{it\rho^2}$ to obtain the kernel $e^{-it\LL_{{\A},a}}(x,y)$ 
\begin{equation}\label{funct-S}
e^{-it\LL_{{\A},a}}(x,y)=K(r_1,\hat{x},r_2,\hat{y})=\sum_{k\in\N}\psi_{k}(\hat{x})\overline{\psi_{k}(\hat{y})}K_{\nu_k}(t; r_1,r_2),
\end{equation}
and
\begin{equation}\label{equ:knuk}
  K_{\nu_k}(t; r_1,r_2)=(r_1r_2)^{-\frac{n-2}2}\int_0^\infty e^{-it\rho^2} J_{\nu_k}(r_1\rho)J_{\nu_k}(r_2\rho) \,\rho d\rho.
\end{equation}
By using the Weber identity (see e.g. \cite[Proposition 8.7]{Taylor}), we have
\begin{align}\label{equ:knukdef12sch}
  K_{\nu}(t,r_1,r_2)=&(r_1r_2)^{-\frac{n-2}2}\int_0^\infty e^{-it\rho^2}J_{\nu}(r_1\rho)J_{\nu}(r_2\rho) \,\rho d\rho\\\nonumber
  =&(r_1r_2)^{-\frac{n-2}2}\lim_{\epsilon\searrow0}\int_0^\infty e^{-(\epsilon+it)\rho^2}J_{\nu}(r_1\rho)J_{\nu}(r_2\rho) \,\rho d\rho\\\nonumber
  =&(r_1r_2)^{-\frac{n-2}2}\lim_{\epsilon\searrow0}\frac{e^{-\frac{r_1^2+r_2^2}{4(\epsilon+it)}}}{2(\epsilon+it)} I_\nu\Big(\frac{r_1r_2}{2(\epsilon+it)}\Big),
\end{align}
where $I_\nu$ is the modified Bessel function of the first kind. 
For $z=\frac{r_1r_2}{2(\epsilon+it)}$ with $\epsilon>0$, we use the integral representation in \cite{Watson}  to write the modified Bessel function $I_\nu$:
\begin{equation}\label{m-bessel}
I_\nu(z)=\frac1{\pi}\int_0^\pi e^{z\cos s} \cos(\nu s) ds-\frac{\sin(\nu\pi)}{\pi}\int_0^\infty e^{-z\cosh s} e^{-s\nu} ds.
\end{equation}
Recalling $\nu=\nu_k=\sqrt{\mu_k+(n-2)^2/4}$, the square root of the eigenvalue of the operator $P=L_{{\A},a}+(n-2)^2/4$,  and using the spectral theory, as well as \cite{Taylor}, we have
\begin{equation}\label{FA}
F(\sqrt{P})=\sum_{k\in\N}\psi_{k}(\hat{x})\overline{\psi_{k}(\hat{y})} F(\nu_k),
\end{equation}
where $F$ is a Borel measurable function. 
Therefore, from \eqref{funct-S},
it follows 
\begin{equation*}
\begin{split}
e^{-it\LL_{{\A},a}}(x,y)&=K(r_1,\hat{x},r_2,\hat{y})=(r_1r_2)^{-\frac{n-2}2}\frac{e^{-\frac{r_1^2+r_2^2}{4it}}}{2it}\sum_{k\in\N}\psi_{k}(\hat{x})\overline{\psi_{k}(\hat{y})} I_\nu\Big(\frac{r_1r_2}{2it}\Big) \\
&=\big(r_1 r_2\big)^{-\frac{n-2}2}\frac{e^{-\frac{r_1^2+r_2^2}{4it}}}{2it}
  \Big(\frac1{\pi}\int_0^\pi e^{\frac{r_1r_2}{2it} \cos(s)} \cos(s\sqrt{P})(\hat{x}, \hat{y}) ds\\
  &\qquad-\frac{\sin(\pi\sqrt{P})}{\pi}\int_0^\infty e^{-\frac{r_1r_2}{2it} \cosh s} e^{-s\sqrt{P}}(\hat{x}, \hat{y}) ds\Big),
\end{split}
\end{equation*}
which gives the desirable expression \eqref{S-kernel}.

  \end{proof}

\section{The localized parametrix construction} \label{sec: parametrix}
In our previous paper \cite{JZ}, under the assumption that the conjugate radius $\conR$ of a compact manifold $Y$ satisfies $\conR>\pi$,
we constructed the global parametrix (essentially the Hadamard parametrix construction) for the even wave propagator $\cos(s\sqrt{P})$ and the Poisson wave propagator $e^{(-s+i\pi)\sqrt{P}}$ when $0\leq s \leq \pi$. However, in the current situation, where $Y=\mathbb{S}^{n-1}$ with conjugate radius $\conR=\pi$, 
the form of the parametrix with only one parameter in the oscillatory interagl with only one parameter no longer applies. 
Fortunately, the localized parametrix construction of the propagator is sufficient for 
establishing Strichartz estimates, as inspired by Hassell and the last author \cite{HZ}.
Therefore, instead of the global parametrix, we will construct the localized parametrix. 

We will use a partition of unity that is subordinate to the covering
\begin{align}
\mathbb{S}^{n-1} = \cup_{j=1}^{2^n} \mathcal{U}_j,
\end{align}
where each $\mathcal{U}_j$ is open neighborhood of the part of the sphere in each quadrant such that any two points in it are connected by a unique distance minimizing geodesic which is part of a great circle and has length less than $\frac{3\pi}{4}$ (any number in $(\frac{\pi}{2},\pi)$ works equally well).
Let $\{Q_j\}_{j=1}^{2^n}$ be a partition of unity
\begin{equation}\label{Id-p-Q}
\mathrm{Id}=\sum_{j=1}^{2^n} Q_j
\end{equation}
each $Q_j$ subordinates to $\mathcal{U}_j$ in this covering, then we have the following localized parametrix.

\begin{lemma}[Hadamard parametrix I] 
\label{lemma: parametrix 1}
Let $d_h=d_h(\hat{x},\hat{y})$ be the distance between two points $\hat{x}, \hat{y}\in \mathbb{S}^{n-1}$, which is smooth on $\mathrm{supp}Q_j \times \mathrm{supp}Q_j$ for each $Q_j$ with $1\leq j\leq 2^n$ given in \eqref{Id-p-Q}. 
Then for $N>n+2$, the kernel of  $ Q_j\cos(s \sqrt{P}) Q_j$ can be written as
 \begin{equation}\label{KR}
\big[Q_j \cos(s \sqrt{P})Q_j \big](\hat{x},\hat{y})=K_N(s; \hat{x},\hat{y})+R_N(s; \hat{x},\hat{y}),
 \end{equation}
 where $R_N(s; \hat{x},\hat{y})\in C^{N-n-2} ([0,\pi]\times \mathbb{S}^{n-1}\times \mathbb{S}^{n-1})$ and  
  \begin{equation}\label{KN1}
  \begin{split}
K_N(s; \hat{x},\hat{y})&=(2\pi)^{n-1}\int_{\R^{n-1}} e^{i d_h(\hat{x},\hat{y}){\bf 1}\cdot\xi} a(s, \hat{x},\hat{y}; |\xi|) \cos(s |\xi|) d\xi\\
&=\sum_{\pm}\int_0^\infty b_{\pm}(\rho d_h) e^{\pm i \rho d_h} a(s, \hat{x},\hat{y}; \rho) \cos(s \rho) \rho^{n-2} d\rho
\end{split}
 \end{equation}
 with ${\bf 1}=(1,0,\ldots,0)$ and $a\in S^0$: 
 \begin{equation}\label{a}
 |\partial^\alpha_{s,\hat{x},\hat{y}}\partial_\rho^k a(s,\hat{x},\hat{y};\rho)|\leq C_{\alpha,k}(1+\rho)^{-k},
 \end{equation}
 and
 \begin{equation}\label{b+-}
\begin{split}
| \partial_r^k b_\pm(r)|\leq C_k(1+r)^{-\frac{n-2}2-k},\quad k\geq 0.
\end{split}
\end{equation}
In addition, we can take $a(s,\hat{x},\hat{y};\rho)$ to be supported in $\rho \geq 1$ and $s \leq \pi-\delta$ for some $\delta>0$.

\end{lemma}

\textcolor{blue}{\begin{remark}
The $\{Q_j\}$ is introduced to separate the conjugate points and our point pairs $(\hat{x}, \hat{y})$ within $\mathrm{supp} Q_j \times \mathrm{supp} Q_j$ are connected by a unique geodesic realizing $d_h(\hat{x},\hat{y})$ due to the fact that injective radius of $\mathbb{S}^{n-1}$ equals to $\pi$.
\end{remark}
}
\textcolor{blue}{\begin{remark}
Intuitively, by the H\"ormander's wavefront set bound of Fourier integral operators, 
the wave front set of $Q_j \cos(s \sqrt{P})Q_j$ is contained in the propagating Lagrangians. In other word, the kernel of $Q_j \cos(s \sqrt{P})Q_j$ is smooth at a given
$(s,\hat{x}, \hat{y})\in \R\times \mathbb{S}^{n-1}\times \mathbb{S}^{n-1}$ unless there is a $\mu_1\in T^*_{\hat{x}} \mathbb{S}^{n-1}$ and a $\mu_2\in T^*_{\hat{y}} \mathbb{S}^{n-1}$ so that $\hat{x}, \hat{y} \in \text{supp}\, Q_j$
and either $(\hat{y},\mu_2) = \exp( s\mathsf{H}_p)(\hat{x},\mu_1) $ or $(\hat{y},\mu_2) = \exp( -s\mathsf{H}_p)(\hat{x},\mu_1) $. Thus, if $(\hat{x},\hat{y}) \in \mathrm{supp}Q_j \times  \mathrm{supp}Q_j$ are sufficiently close, then the parameter
 $s$  must either be small or near $2\pi$ (the length of closed geodesics); otherwise, the kernel becomes smooth.
\end{remark}
}

\begin{proof}

 The interpretation of $\cos(s\sqrt{P})$ in \cite[Section~3]{JZ} as a sum of two Fourier integral operators associated to forward (`+' sign below) and backward (`-' sign below) propagating Lagrangians is still valid in our current setting. By propagating Lagrangians, we mean:
 \begin{align}  \label{eq: propagating lagrangians}
\begin{split}
\mathscr{L}_{\pm}:= & \{ (s,\hat{y},\hat{x},\tau,\mu_2,-\mu_1) \in T^*(\mathbb{R} \times Y \times Y): 
\\& \tau = \mp |\mu_1|, (\hat{y},\mu_2) = \exp(\pm s\mathsf{H}_p)(\hat{x},\mu_1) \},
\end{split}
\end{align}
where $p = |\mu|^2$ is the homogeneous principal symbol of $P$, and
\begin{equation}
\mathsf{H}_{p} = (2|\mu|)^{-1}H_{p}
\end{equation}
is the rescaled Hamilton vector field.

The only potential issue that could be caused by the presence of conjugate points is that the phase functions
 \begin{equation}
\phi_{\pm} = (\xi \cdot \textbf{1})d_h(\hat{x},\hat{y}) \mp s|\xi|
\end{equation}
can't be used to parametrize $\mathscr{L}_\pm$ in the sense of \cite{FIO1} anymore. However, this issue is overcome again the localizer $Q_j$: when $(\hat{x},\hat{y}) \in \mathrm{supp}Q_j \times  \mathrm{supp}Q_j$ (i.e., on the support of $\tilde{K}_N$), the geodesic normal coordinate centred at $\hat{x}$ is still valid (since there are no conjugate point pairs over this region), thus , we can still use $\phi_\pm$ the parametrize $\mathscr{L}_\pm$, and the result follows from the same proof as in \cite[Section~3]{JZ}.

In addition, our current operator $P$ has an extra first order perturbation $i{\A}(\hat{x})\cdot\nabla_{\mathbb{S}^{n-1}}$ compared with \cite[Section~3]{JZ}, but this does not affect the parametrix construction as well, since the construction of $e^{\pm is\sqrt{P}}$ is via the general theory of Fourier integral operators, which is robust under such perturbations.

Finally, we don't have the complication about the distance spectrum there because our point pairs within $\mathrm{supp} Q_j \times \mathrm{supp} Q_j$ are connected by a unique geodesic realizing $d_h(\hat{x},\hat{y})$. And we can impose the claimed support condition on $\rho$ because we can insert a cut-off in $\rho$ and move the low-frequency part, which is a smooth function, into $R_N$. 
For the support condition on $s$, we can insert a cut-off $\chi(s)$ on the amplitude that is identically one in $[0,\pi-2\delta]$ and supported in $[0,\pi-\delta]$ with $\delta$ sufficiently small.
The error term with amplitude $(1-\chi(s))a(s,\hat{x},\hat{y};\rho)$ is in $C^{\infty}(\R \times \mathbb{S}^{n-1} \times \mathbb{S}^{n-1})$.
This is because the wave-front set of $\cos(s\sqrt{P})$ is included in the propagating Lagrangians.

We notice that for $\delta>0$ sufficiently small, there is no geodesic connecting points in $\mathrm{supp} Q_j \times \mathrm{supp} Q_j$ with length in $[\pi-2\delta,\pi]$. So this part has empty wave-front set and is in $C^\infty(\R \times \mathbb{S}^{n-1} \times \mathbb{S}^{n-1})$. Consequently, we can put this part of the oscillatory integral into the $R_N$-term.
\end{proof}


Next we show that $e^{(-s\pm i\pi))\sqrt{P}}$ can be written as the same type of oscillatory integral as above and is residual after inserting $Q_j$ on both sides.

 \begin{lemma}[Poisson-wave operator] 
\label{lemma: Poisson-wave}
Let $d_h=d_h(\hat{x},\hat{y})$ be the distance between two points $\hat{x}, \hat{y}\in \mathbb{S}^{n-1}$.
Then for each $Q_j$, $1\leq j\leq 2^n$ as above, and $\forall N>n+2$, $s\geq 0$, the kernel of the localized Poisson-wave operator $Q_j e^{(-s\pm i\pi))\sqrt{P}}Q_j$ is smooth:
\begin{equation}
Q_j(\hat{x})\tilde{K}Q_j(\hat{y}) \in C^\infty([0,\infty) \times \CS \times \CS).
\end{equation}

\end{lemma}

\begin{proof} 

The smoothness in $s$ follows directly if one shows that it is $C^\infty$ with respect to $\hat{x},\hat{y}$, since differentiating in $s$ only adds a $\sqrt{P}$ factor.
For $s=0$, this follows from the property of
\begin{equation}
Q_je^{i\pi\sqrt{P}}Q_j,
\end{equation}
by the same discussion as above: $e^{i\pi\sqrt{P}}$ has canonical relation over points $(\hat{x},\hat{y})$ such that $d_h(\hat{x},\hat{y})=\pi$, which is disjoint from $\mathrm{supp} Q_j \times \mathrm{supp} Q_j$. 
For $s>0$, the phase function has exponential decay in $\rho$ (or $|\xi|$) already, which means that one can view this as an oscillatory integral with amplitude in $S^{-\infty}$ and this gives the desired smoothness.

\end{proof}

\section{The Littlewood-Paley theory 
associated with the operator $\LL_{{\A},a}$}\label{sec:LP}

In this section,  we study the the Bernstein inequalities and the square function inequalities associated with the Schr\"odinger operator $\LL_{{\A},a}$ that are needed in the proof of our Strichartz estimates. 

For this purpose, we introduce $\varphi\in C_c^\infty(\mathbb{R}\setminus\{0\})$, with $0\leq\varphi\leq 1$, $\text{supp}\,\varphi\subset[3/4,8/3]$, and
\begin{equation}\label{LP-dp}
\sum_{j\in\Z}\varphi(2^{-j}\lambda)=1,\quad \varphi_j(\lambda):=\varphi(2^{-j}\lambda), \, j\in\Z,\quad \phi_0(\lambda):=\sum_{j\leq0}\varphi(2^{-j}\lambda).
\end{equation}
More precisely, we prove the following propositions. 
\begin{proposition}[Bernstein inequalities]\label{prop:Bern}
Let $\varphi(\lambda)$ be a $C^\infty_c$ bump function on $\R$  with support in $[\frac{1}{2},2]$ and let $\alpha$ and $p(\alpha)$ be given in \eqref{def:alpha} and \eqref{def:q-alpha} respectively, then it holds for any $f\in L^q(\R^n)$ and $j\in\mathbb{Z}$
\begin{equation}\label{est:Bern}
\|\varphi(2^{-j}\sqrt{\LL_{{\A},a}})f\|_{L^p(\R^n)}\lesssim2^{nj\big(\frac{1}{q}-\frac{1}{p}\big)}\|\varphi(2^{-j}\sqrt{\LL_{{\A},a}}) f\|_{L^q(\R^n)}, \,
\end{equation}
provided $p'(\alpha)<q\leq p<p(\alpha)$.
In addition, if $\alpha\geq0$, the range can be extended to $1\leq q< p\leq +\infty$ including the endpoints.
\end{proposition}

\begin{proposition}[The square function inequality]\label{prop:squarefun} Let $\{\varphi_j\}_{j\in\mathbb Z}$ be a Littlewood-Paley sequence given by \eqref{LP-dp} and let $\alpha$ and $p(\alpha)$ be given in \eqref{def:alpha} and \eqref{def:q-alpha} respectively.
Then, for $p'(\alpha)<p<p(\alpha)$,
there exist constants $c_p$ and $C_p$ depending on $p$ such that
\begin{equation}\label{square}
c_p\|f\|_{L^p(\R^n)}\leq
\Big\|\Big(\sum_{j\in\Z}|\varphi_j(\sqrt{\LL_{{\A},a}})f|^2\Big)^{\frac12}\Big\|_{L^p(\R^n)}\leq
C_p\|f\|_{L^p(\R^n)}.
\end{equation}

\end{proposition}

The Littlewood-Paley theory, which is often associated with heat kernel estimates, has its own intrinsic interest. 
The Littlewood-Paley theory for the Schr\"odinger operator with a purely electric inverse-square potential was studied by Killip, Miao, Visan, Zheng and the last author \cite{KMVZZ}, where the starting point is the pointwise estimate of the heat kernel. In \cite[Proposition 5.1, 5.2] {JZ}, we provide an alternative argument to prove analogues of these results, relying on the heat kernel estimates 
\begin{equation}\label{up-est:heat-c}
\begin{split}
\big|e^{-tH}(r_1,y_1;r_2,y_2)\big|\leq C \Big[\min\Big\{1,\Big(\frac{r_1r_2}{2t}\Big)\Big\}\Big]^{\alpha} t^{-\frac{n}2}e^{-\frac{d^2((r_1,y_1),(r_2,y_2))}{c t}},
  \end{split}
  \end{equation}
which was proved in \cite[Theorem 1.1]{HZ23}, where $H=-\Delta_g+V_0(y)r^{-2}$ on metric cone $C(Y)=(0,+\infty)\times Y$.\vspace{0.1cm}

In the current case, it is worthwhile to mention that the Schr\"odinger operator $\LL_{{\A},a}$ in \eqref{LAa} is perturbed by both magnetic and electric singular potentials.
Since $A(x)={\A}(\hat{x})|x|^{-1}\in L^2_{\text{loc}}(\R^n)$ when $n\geq 3$, it is tempting to use the Simon’s diamagnetic pointwise inequality  (see e.g.  \cite[Theorem B.13.2]{Simon}, \cite{AHS1})
\begin{equation}\label{est:Simon}
\big|e^{-t[(i\nabla+A(x))^2+V(x)]}f \big|\leq C  e^{-t(-\Delta+V(x))}|f|,
\end{equation}
and the special case of \eqref{up-est:heat-c} with $Y=\mathbb{S}^{n-1}$
\begin{equation}\label{up-est:heat-c'}
\begin{split}
\big|e^{-t\LL_{0,a}}(r_1,\hat{x};r_2, \hat{y})\big|\leq C \Big[\min\Big\{1,\Big(\frac{r_1r_2}{2t}\Big)\Big\}\Big]^{\alpha} t^{-\frac{n}2}e^{-\frac{|x-y|^2}{c t}}.
  \end{split}
  \end{equation}
Unfortunately, the inverse-square potential $V(x)=a(\hat{x})|x|^{-2}$ does not belong to the Kato class \cite{DFVV,Simon} and does not satisfy the condition in \cite[Theorem B.13.2]{Simon},
so we do not know whether \eqref{est:Simon} is valid for our case. 
But if we could prove the heat kernel estimates 
\begin{equation}\label{up-est:heat}
\begin{split}
\big|e^{-t\LL_{{\A},a}}(r_1,\hat{x};r_2, \hat{y})\big|\leq C \Big[\min\Big\{1,\Big(\frac{r_1r_2}{2t}\Big)\Big\}\Big]^{\alpha} t^{-\frac{n}2}e^{-\frac{|x-y|^2}{c t}},
  \end{split}
  \end{equation}
then we can prove  the Bernstein inequalities and the square function inequalities associated with the Schr\"odinger operator $\LL_{{\A},a}$ as in \cite{JZ}.
Therefore, to prove Proposition \ref{prop:Bern} and Proposition \ref{prop:squarefun}, it suffices to show \eqref{up-est:heat}. 

For this purpose, we follow the argument of \cite{HZ23} by Huang and the last author. We here only sketch the main steps and modifications needed, but refer readers to \cite{HZ23} for details. Recalling \eqref{LAa-r}, we have
\begin{equation*}
\begin{split}
\mathcal{L}_{{\A},a}=-\partial_r^2-\frac{n-1}r\partial_r+\frac{L_{{\A},a}}{r^2},
\end{split}
\end{equation*}
with $L_{{\A},a}$ in \eqref{L-angle}. So we take $Y=\mathbb{S}^{n-1}$ and $P=\sqrt{L_{{\A},a}+(n-2)^2/4}$ in \cite{HZ23}, the minor difference is from the magnetic potential.
As in the proof of Proposition \ref{prop:Sch-pro}, we have the heat kernel
  \begin{equation*}
\begin{split}
e^{-t\LL_{{\A},a}}(x,y)&=(r_1r_2)^{-\frac{n-2}2}\frac{e^{-\frac{r_1^2+r_2^2}{4t}}}{2t}\sum_{k\in\N}\psi_{k}(\hat{x})\overline{\psi_{k}(\hat{y})} I_{\nu_k}\Big(\frac{r_1r_2}{2t}\Big) \\
&=\big(r_1 r_2\big)^{-\frac{n-2}2}\frac{e^{-\frac{r_1^2+r_2^2}{4t}}}{2t}
  \Big(\frac1{\pi}\int_0^\pi e^{\frac{r_1r_2}{2t} \cos(s)} \cos(s\sqrt{P})(\hat{x}, \hat{y}) ds\\
  &\qquad-\frac{\sin(\pi\sqrt{P})}{\pi}\int_0^\infty e^{-\frac{r_1r_2}{2t} \cosh s} e^{-s\sqrt{P}}(\hat{x}, \hat{y}) ds\Big).
\end{split}
\end{equation*}
Then, by using the similar argument for \cite[Equation~(3.1)]{HZ23}, \eqref{up-est:heat} follows from:
\begin{lemma}\label{lem:key}  Let $\alpha=-\frac{n-2}2+\nu_0$ and $\delta=d_h(\hat{x},\hat{y})$, then
there exist positive constants $C$ and $N$ only depending on $n$ such that 

$\bullet$ either for $0<\frac{r_1r_2}{2t}\leq 1$,
\begin{equation}\label{est:up<}
\begin{split}
\Big| \Big(\frac{r_1r_2}{2t}\Big)^{-\frac{n-2}2}\sum_{k\in\N}\psi_{k}(\hat{x})\overline{\psi_{k}(\hat{y})} I_{\nu_k}\Big(\frac{r_1r_2}{2t}\Big)\Big|\leq
C \Big(\frac{r_1r_2}{2t}\Big)^{\alpha}e^{\big(\frac{r_1r_2}{2t}\big)\cos\delta},
  \end{split}
 \end{equation}

$\bullet$ or for $\frac{r_1r_2}{2t}\gtrsim1$ 
\begin{equation}\label{est:up>}
\begin{split}
\Big| \Big(\frac{r_1r_2}{2t}\Big)^{-\frac{n-2}2}&\sum_{k\in\N}\psi_{k}(\hat{x})\overline{\psi_{k}(\hat{y})} I_{\nu_k}\Big(\frac{r_1r_2}{2t}\Big)\Big|\\
&\leq
C \times
\begin{cases}
e^{\big(\frac{r_1r_2}{2t}\big)\cos\delta}+\big(\frac{r_1r_2}{2t}\big)^N,\quad &0\leq \delta \leq \frac\pi2,\\ 
\big(\frac{r_1r_2}{2t}\big)^N e^{\big(\frac{r_1r_2}{2t}\big)\cos\delta},\quad &\frac\pi2\leq \delta \leq \pi.
\end{cases}
  \end{split}
 \end{equation}

\end{lemma}

\begin{remark}\label{rem:comp-h-s}
As $\frac{r_1r_2}{2t}\to \infty$, the heat kernel estimates allow for a polynomial factor $\big(\frac{r_1r_2}{2t}\big)^N$ in \eqref{est:up>}, which can be absorbed by the exponential decay factor
$\exp\big(-\frac{r_1^2+r_2^2}{4t}\big)$ in the kernel estimate. However, the Schr\"odinger kernel only contains an oscillatory factor $\exp\big(-\frac{r_1^2+r_2^2}{4it}\big)$, with no accompanying decay. 
This is why proving decay estimates for the Schr\"odinger  propagator in \cite{JZ} is significantly more complicated than the heat kernel estimates in \cite{HZ23}.
\end{remark}

This is an analogue of \cite[Lemma 3.1]{HZ23} with injective radius $\conR=\pi$ and $0\leq\delta\leq\pi$ since $Y=\mathbb{S}^{n-1}$. 
We refer readers to \cite{HZ23} for the detailed argument but explain here  modifications needed in that proof. 
The difference between the operators $P$ here and in \cite{HZ23} is a lower-order operator $|{\A}(\hat{x})|^2+2i {\A}(\hat{x})\cdot\nabla_{\mathbb{S}^{n-1}}$. As pointed out in the proof of 
Lemma~\ref{lemma: parametrix 1}, this difference is harmless for the parametrix used in the argument of \cite{HZ23} to treat the case when $0\leq \delta \leq \frac\pi2$. 
We also refer to H\"ormander \cite[Chapter~XVII]{HorVol3}, which shows that such lower-order perturbations are harmless
for the parametrix.

In addition, estimates applied to $\psi_k$ in \cite[Lemma 3.1]{HZ23} relies only on the estimates of eigenvalues of $P$, which can be derived from the Wely's law to the leading order, and the $L^\infty$-bound in \cite[Equation~(3.2.5)]{Sogge-H}, whose proof there holds without modification in our current setting as long as the Weyl's law is verified as well. 
And this desired Weyl's law with perturbations follows from \cite[Equation~(0.6)]{Garding53}.

\section{Localized Dispersive estimates}\label{sec:dispersive}
In this section, we establish dispersive estimates using the stationary phase method and the localized parametrix constructed in the previous section.

\begin{proposition}[Localized pointwise estimates]\label{prop:dispersive}
Let $x=(r_1,\hat{x})$ and $y=(r_2,\hat{y})$ be in $(0,+\infty)\times\mathbb{S}^{n-1}$ of dimension $n\geq 3$ and let $Q_j$ be defined in \eqref{Id-p-Q}. Then, for $t\neq 0$,
the kernel of Schr\"odinger propagator satisfies the properties:

\begin{itemize}
\item When $\frac{r_1r_2}{2|t|}\lesssim 1$, we have the global estimate
\begin{equation}\label{est:dispersive<}
 \begin{split}
\big| e^{it\LL_{{\A},a}} (x,y)\big|\leq C|t|^{-\frac n2}\times \Big(\frac{r_1r_2}{2|t|}\Big)^{-\frac{n-2}2+\nu_0}.
\end{split}
\end{equation}

\item When $\frac{r_1r_2}{2|t|}\gg1$,  for $1\leq j\leq 2^n$, we have the localized estimates 
\begin{equation}\label{est:dispersive}
 \begin{split}
\big| \big[Q_j e^{it\LL_{{\A},a}} Q_j\big] (x,y)\big|\leq C|t|^{-\frac n2}
\end{split}
\end{equation}
for a constant $C$ that is independent of $x,y\in(0,+\infty)\times\mathbb{S}^{n-1}$. 
\end{itemize}
Here $\nu_0:=\sqrt{\mu_1+(n-2)^2/4}$ is the positive square roof of the smallest eigenvalue of the positive operator
$P=L_{{\A},a}+(n-2)^2/4$ on the sphere $\mathbb{S}^{n-1}$.
\end{proposition}

As a consequence of Proposition \ref{prop:dispersive}, we have the localized decay estimates:
\begin{proposition}\label{prop:LqLq'} Let $\alpha$ and $p(\alpha)$ be \eqref{def:alpha} and \eqref{def:q-alpha} respectively. Then,  for $t\neq 0$, there exists a constant $C$ such that,
either if $\alpha\geq0$, 
\begin{equation}\label{est:LqLq'}
\big\| Q_j e^{it\LL_{{\A},a}} Q_j\big\|_{L^{p'}(\R^n)\to L^p(\R^n)}\leq C |t|^{-\frac n2(1-\frac2p)},\quad p\in [2, +\infty];
\end{equation}
or if $-\frac{n-2}2<\alpha <0$, 
\begin{equation}\label{est:LqLq'0}
\big\| Q_j e^{it\LL_{{\A},a}} Q_j\big\|_{L^{p'}(\R^n)\to L^p(\R^n)}\leq C |t|^{-\frac n2(1-\frac2p)},\quad p\in [2, p(\alpha)).
\end{equation}

\end{proposition}

The proofs of these two propositions are modified from our previous paper \cite{JZ}. We sketch the main ideas and steps here, and refer the reader to \cite[Section~4]{JZ} for the detailed arguments.

\begin{proof}[The sketch proof of Proposition \ref{prop:dispersive}]
As stated in Proposition \ref{prop:dispersive}, we consider two cases:  $\frac{r_1r_2}{2|t|}\lesssim 1$ and $\frac{r_1r_2}{2|t|}\gg 1$. The first case can be treated using the spectral theory and
properties of the spectrum. The second case is more challenging and will be addressed using the stationary phase method and the localized Hadamard parametrix.\vspace{0.1cm}

{\bf Case I: $\frac{r_1r_2}{2|t|}\lesssim 1$.} In this case, we could prove
\begin{equation*}
\begin{split}
\big| e^{it\LL_{{\A},a}}(x,y)\big|&=|t|^{-\frac n2}\Big(\frac{r_1r_2}{|t|}\Big)^{-\frac{n-2}2}\Big|\sum_{k\in\N}\psi_{k}(\hat{x})\overline{\psi_{k}(\hat{y})} J_{\nu_k}\Big(\frac{r_1r_2}{2t}\Big)\Big| \\
&\leq C |t|^{-\frac n2}\Big(\frac{r_1r_2}{|t|}\Big)^{-\frac{n-2}2+\nu_0}
\end{split}
\end{equation*}
by using the aforementioned Weyl's asymptotic formula \cite[Equation~(0.6)]{Garding53} and the asymptotic estimates of eigenfunction 
 \begin{equation}\label{est:eig}
\nu^2_k\sim (1+k)^{\frac 2{n-1}},\quad k\geq 1,\implies\|\psi_{k}(\hat{x})\|^2_{L^\infty(\mathbb{S}^{n-1})}\leq C(1+\nu^2_k)^{\frac{n-2}2}\leq C (1+k)^{\frac{n-2}{n-1}},
 \end{equation}
and the estimate of Bessel function
\begin{equation}\label{est:r}
|J_\nu(z)|\leq\frac{Cz^\nu}{2^\nu\Gamma(\nu+\frac{1}{2})\Gamma(\frac12)}\Big(1+\frac{1}{\nu+\frac12}\Big).
\end{equation}
We refer to \cite[Proposition 4.1]{JZ} for the details.\vspace{0.2cm}
  
{\bf Case II: $\frac{r_1r_2}{2|t|}\gg 1$.} In this case, by using Proposition \ref{prop:Sch-pro}, it suffices to prove 
\begin{equation}\label{est:>1}
\begin{split}
\Big|\Big(\frac{r_1 r_2}{|t|}\Big)^{-\frac{n-2}2}
 Q_j \Big(&\frac1{\pi}\int_0^\pi e^{\frac{r_1r_2}{2it} \cos(s)} \cos(s\sqrt{P})(\hat{x}, \hat{y}) ds\\
  &-\frac{\sin(\pi\sqrt{P})}{\pi}\int_0^\infty e^{-\frac{r_1r_2}{2it} \cosh s} e^{-s\sqrt{P}}(\hat{x}, \hat{y}) ds\Big)Q_j\Big| \leq C.
\end{split}
\end{equation}

Let $z:=\frac{r_1r_2}{2|t|}\gg 1$, we decompose the kernel into two terms, the propagation term:
\begin{equation}\label{S1}
\begin{split}
I_P(z;\hat{x}, \hat{y}):=\frac{z^{-\frac{n-2}2}}{\pi}\int_0^\pi  e^{-iz\cos(s)}  Q_j K_N(s,y_1,y_2) Q_j ds,
\end{split}
\end{equation}
and the residual term
\begin{equation}\label{S3}
\begin{split}
I_{R}(z;\hat{x}, \hat{y}):=& 
\frac{z^{-\frac{n-2}2}}{\pi} Q_j \big( \int_0^\pi  e^{-iz\cos(s)}  R_N(s,y_1,y_2) ds
\\& -\sin(\pi\sqrt{P})\int_0^\infty  e^{iz\cosh(s)} e^{-s\sqrt{P}} ds \big)Q_j.
\end{split}
\end{equation}
Here $P$ stands for propagation and $R$ stands for residual.

The estimate \eqref{est:>1} is proved if we could prove that both 
$I_{P}(z;\hat{x}, \hat{y})$ and $I_{R}(z;\hat{x}, \hat{y})$ are uniformly bounded when $z\gg1$.

\begin{proof}[The contribution of  $I_P$] 
Since we are localized to $0\leq s\leq \pi-\delta$ in $I_P(z; \hat{x}, \hat{y})$ and our oscillatory integral has exactly the same form as that in \cite[Lemma 4.3]{JZ}. So \cite[Lemma 4.3]{JZ} gives
\begin{equation}\label{est:>1-1}
\begin{split}
\big| I_P(z; \hat{x}, \hat{y})\big| \leq C
\end{split}
\end{equation}
since $\|Q_j\|_{L^\infty(\mathbb{S}^{n-1})} \leq 1$.

\end{proof}

\bigskip

\begin{proof}[The contribuiton from $I_R$] 
To estimate $I_{R}(z;\hat{x}, \hat{y})$, we split it into two terms
\begin{equation}\label{S3-1}
\begin{split}
I_{R,1}(z;\hat{x}, \hat{y}):=& 
\frac{z^{-\frac{n-2}2}}{\pi} Q_j \big( \int_0^\pi  e^{-iz\cos(s)}  R_N(s,y_1,y_2) ds
\\& -\sin(\pi\sqrt{P})\int_0^1  e^{iz\cosh(s)} e^{-s\sqrt{P}} ds \big)Q_j,
\end{split}
\end{equation}
and 
\begin{equation}\label{S3-2}
\begin{split}
I_{R,2}(z;\hat{x}, \hat{y}):= -
\frac{z^{-\frac{n-2}2}}{\pi} Q_j \Big(\sin(\pi\sqrt{P})\int_1^{+\infty}  e^{iz\cosh(s)} e^{-s\sqrt{P}} ds \Big)Q_j.
\end{split}
\end{equation}
For the first term, as we have shown in Lemma~\ref{lemma: parametrix 1} and Lemma~\ref{lemma: Poisson-wave}, this part has $C^{N-n-2}$-(hence uniformly bounded) kernel, and we have
\begin{equation*}
\begin{split}
\big| I_{R,1}(z;\hat{x}, \hat{y})\big|&\lesssim  
z^{-\frac{n-2}2}  \big( \int_0^\pi  |R_N(s,y_1,y_2)| ds
+\int_0^1  |R_N(s,y_1,y_2)| ds \big) \lesssim 1.
\end{split}
\end{equation*}
To estimate the second term, since $\|Q_j\|_{L^\infty(\mathbb{S}^{n-1})} \leq 1$, it is enough to show
\begin{equation*}
\begin{split}
z^{-\frac{n-2}2}\Big|\sum_{k\in\N}\psi_{k}(\hat{x})\overline{\psi_{k}(\hat{y})} \int_1^{+\infty}  e^{iz\cosh(s)} e^{-s\nu_k\pm i\pi\nu_k} ds \Big| \lesssim 1.
\end{split}
\end{equation*}
By \eqref{est:eig} and $z\gg1$, the left hand side of the above is bounded by
\begin{equation*}
\begin{split}
z^{-\frac{n-2}2}&\sum_{k\in\N}(1+\nu_k^2)^{\frac{n-2}2} \int_1^{+\infty}   e^{-s\nu_k} ds \\
&\lesssim z^{-\frac{n-2}2}\sum_{k\in\N}(1+\nu_k^2)^{\frac{n-2}2} e^{-\frac{\nu_k}2} \int_1^{+\infty}   e^{-\frac{s\nu_0}2} ds\lesssim_{\nu_0} 1.
\end{split}
\end{equation*}

\end{proof}

\end{proof}

\begin{proof}[The sketch proof of Proposition \ref{prop:LqLq'} ]
The proof is proceeds in the same manner as in our previous paper \cite[Section~6]{JZ}. By the spectral theorem, one has the $L^2$-estimate
\begin{equation}\label{est:L2q}
\|Q_j e^{it\LL_{{\A},a}}Q_j\|_{L^2(\R^n)\to L^2(\R^n)}
\leq \|e^{it\LL_{{\A},a}}\|_{L^2(\R^n)\to L^2(\R^n)}\leq C,
\end{equation}
which can be proved by using the unitary property of the Hankel transform. For instance, we can see this from the argument of \eqref{est:L2} below. 
So, if $\alpha\geq0$, we obtain \eqref{est:LqLq'} by interpolating \eqref{est:L2q} and 
$$\|Q_j e^{it\LL_{{\A},a}}Q_j\|_{L^1(\R^n) \to L^\infty (\R^n)} \leq C|t|^{-\frac n2},$$
which is a consequence of \eqref{est:dispersive} and \eqref{est:dispersive<}.\vspace{0.2cm}

If $\alpha<0$, one cannot obtain \eqref{est:LqLq'0}  by interpolation as above. To prove \eqref{est:LqLq'0}, we have to strengthen the result obtained from interpolation by getting rid of the weight when $p\in [2, p(\alpha))$. \vspace{0.2cm}

To this end, recalling \eqref{hk}, we first introduce the orthogonal projections on $L^{2}$
\begin{equation}\label{def-pro1}
  P_k:
  L^{2}(\R^n) \to L^{2}(r^{n-1}dr)\otimes  h_{k}(\mathbb{S}^{n-1}),
\end{equation}
and
\begin{equation}\label{def-pro}
  P_<:
  L^{2}(\R^n)\to  
  \bigoplus_{\{k\in \mathbb{N}: \nu_k< (n-2)/2\}}   L^{2}(r^{n-1}dr)\otimes  h_{k}(\mathbb{S}^{n-1}),
  \quad
  P_{\geq }=I-P_{<}.
\end{equation}
Here the space $h_{k}(\mathbb{S}^{n-1})$ is the linear
span of 
$\{\psi_k(\hat{x})\}$
defined in \eqref{equ:eig-Aa}. Then we can decompose the Schr\"odinger propagator as
\begin{equation}\label{d-propag}
e^{it\LL_{{\A},a}}f=
 e^{it\LL_{{\A},a}}P_{<}f+
 e^{it\LL_{{\A},a}}P_{\geq }f.
\end{equation}
From \eqref{funct-S}, we can write
\begin{equation}\label{ker:S-l}
\begin{split}
 e^{it\LL_{{\A},a}}P_{<}
&=\big(r_1 r_2\big)^{-\frac{n-2}2}\sum_{\{k\in\mathbb{N}: \nu_k<(n-2)/2\}}\psi_{k}(\hat{x})\overline{\psi_{k}(\hat{y})}K_{\nu_k}(t,r_1,r_2),
\end{split}
\end{equation}
and 
\begin{equation}\label{ker:S-h}
\begin{split}
 e^{it\LL_{{\A},a}}P_{\geq }
&=\big(r_1 r_2\big)^{-\frac{n-2}2}\sum_{\{k\in\mathbb{N}: \nu_k\geq \frac12(n-2)\}}\psi_{k}(\hat{x})\overline{\psi_{k}(\hat{y})}K_{\nu_k}(t,r_1,r_2).
\end{split}
\end{equation}
Since the kernel $ e^{it\LL_{{\A},a}}P_{\geq }$ only has contribution from large angular momenta, thus we can repeat the proof of Proposition \ref{prop:dispersive} to show 
$$\big| Q_j e^{it\LL_{{\A},a}}P_{\geq} Q_j\big|\leq C |t|^{-\frac n2}.$$
Therefore, similar to the case where $\alpha\geq 0$, we can prove \eqref{est:LqLq'0}  for $ Q_j e^{it\LL_{{\A},a}}P_{\geq } Q_j$ with $q\geq2$. Thus we are left to consider $ Q_j e^{it\LL_{{\A},a}}P_{<} Q_j$, where we are restricted to small angular momenta. In this small angular momenta case, we can drop $Q_j$ to prove global estimate instead:
\begin{equation}\label{est:LqLq'0s}
\big\| e^{it\LL_{{\A},a}} P_{<} \big\|_{L^{p'}(\R^n)\to L^p(\R^n)}\leq C |t|^{-\frac n2(1-\frac2p)},\quad p\in [2, p(\alpha)).
\end{equation}
Due to the aforementioned Weyl’s asymptotic formula \cite[Equation~(0.6)]{Garding53}, which implies
$$
\nu^2_k\sim (1+k)^{\frac 2{n-1}},\quad k\geq 1,
$$ 
the summation in the kernel $ e^{it\LL_{{\A},a}}P_{<}$ in \eqref{ker:S-l} has only finitely many terms. 
Hence, to prove \eqref{est:LqLq'0s} for  $ e^{it\LL_{{\A},a}}P_{<}$,
we only need to prove \eqref{est:LqLq'0s} for  $ e^{it\LL_{{\A},a}}P_{k}$ with each $k$ such that $\nu_k<(n-2)/2$. \vspace{0.1cm}

By using the Littlewood-Paley square function inequalities \eqref{square}
and the Minkowski inequalities, it suffices to show
  \begin{equation}\label{est:q-q'-m}
  \Big\|\varphi_j(\sqrt{\LL_{{\A},a}}) e^{it\LL_{{\A},a}}P_{k} f\Big\|_{L^p(\R^n)}\le
    C_k |t|^{-\frac n2(1-\frac 2p)} \Big\|\tilde{\varphi}_j(\sqrt{\LL_{{\A},a}}) P_{k} f\Big\|_{L^{p'}(\R^n)},
  \end{equation}
  provided $p\in [2, p(\alpha))$, where  we choose $\tilde{\varphi}\in C_c^\infty((0,+\infty))$ such that $\tilde{\varphi}(\lambda)=1$ if $\lambda\in\mathrm{supp}\,\varphi$
and $\tilde{\varphi}\varphi=\varphi$. In the following argument, since $\tilde{\varphi}$ shares the same properties as $\varphi$, we drop the tilde over $\varphi$ for brevity.

For $f\in L^2$, we expand 
\begin{equation}\label{f:exp}
\begin{split}
f=\sum_{k\in\mathbb{N}}c_{k}(r)\psi_k(\hat{x}),
\end{split}
\end{equation}
and let $\tilde{c}_k(r)=\varphi_j(\sqrt{\LL_{{\A},a}}) c_k(r)$, similarly to  \eqref{funct-S}, we write
  \begin{equation*}
    \begin{split}
\varphi_j(\sqrt{\LL_{{\A},a}}) e^{it\LL_{{\A},a}}P_{k} f&=\psi_k(\hat{x}) 2^{jn} \int_0^\infty  K^l_{\nu_k}(2^{2j}t;2^jr_{1}, 2^jr_{2}) \tilde{c}_k(r_2)\, r^{n-1}_2 dr_2\\
&=\psi_k(\hat{x}) \big(T_{\nu_k}\tilde{c}_k(2^{-j}r_2)\big)(2^{2j}t, 2^jr_1).
    \end{split}
  \end{equation*}
To estimate it, we recall \cite[Proposition 6.1]{JZ}:
\begin{lemma}\label{prop:est-qq'}
  Let $0 <\nu\leq \frac{n-2}2$ and $\sigma(\nu)=-(n-2)/2+\nu$. Let $T_\nu$ be the operator defined as
  \begin{equation}\label{Tnu-operator}
\begin{split}
(T_{\nu}g)(t,r_1)=\int_0^\infty  K^l_{\nu}(t;r_{1},r_{2}) g(r_2)\, r^{n-1}_2 dr_2
 \end{split}
\end{equation}
 and 
\begin{equation*}
\begin{split}
  K^l_{\nu}(t,r_1,r_2)&=(r_1r_2)^{-\frac{n-2}2}\int_0^\infty e^{it\rho^2}J_{\nu}(r_1\rho)J_{\nu}(r_2\rho) \varphi(\rho)\,\rho d\rho,
  \end{split}
\end{equation*}
where $\varphi$ is given in \eqref{LP-dp}.
Then, for $2\leq q<q(\sigma)$, the following estimate holds
  \begin{equation}\label{est:q-q'}
  \|T_{\nu}g\|_{L^q({r^{n-1}_1 dr_1})}\le
    C_{\nu}|t|^{-\frac n2(1-\frac 2q)}\|g\|_{L^{q'}_{r^{n-1}_2 dr_2}}.
  \end{equation}
\end{lemma}
Notice that $p(\alpha)\leq p(\sigma)$ for $\nu \geq \nu_0$ and $\sigma(\nu)=-(n-2)/2+\nu$, we use this proposition to obtain that 
  \begin{equation*}
  \begin{split}
  &\Big\|\varphi_j(\sqrt{\LL_{{\A},a}})  e^{it\LL_{{\A},a}}P_{k} f\Big\|_{L^p(\R^n)}\le
    C_k \|\big(T_{\nu_k}\tilde{c}_k(2^{-j}\cdot)\big)(2^{2j}t, 2^jr_1)\|_{L^p_{r^{n-1}_1 dr_1}} \\&\le C_{k}|t|^{-\frac n2(1-\frac 2p)}\|\tilde{c}_k(r)\|_{L^{p'}_{r^{n-1} dr}}\le C_k |t|^{-\frac n2(1-\frac 2p)} \Big\|\varphi_j(\sqrt{\LL_{{\A},a}})  P_{k} f\Big\|_{L^{p'}(\R^n)},
    \end{split}
  \end{equation*}
  which shows the desirable estimate \eqref{est:q-q'-m}. 
Therefore, we have completed the proof of Proposition \ref{prop:LqLq'}.
\end{proof}

\section{The proof of Theorem \ref{thm:Stri}} \label{sec: strichartz, H0 level}
In this section, we primarily prove Theorem  \ref{thm:Stri} using Proposition \ref{prop:LqLq'} and a  variant of the abstract Strichartz estimates from Keel-Tao's work \cite{KT}.\vspace{0.2cm}

Let $Q_j$ be defined in \eqref{Id-p-Q} and define 
\begin{equation}\label{def:Uj}
U_j(t)=Q_j e^{it\LL_{{\A},a}},
\end{equation}
then we see that
\begin{equation}\label{U_jsum}
U(t):=e^{it\LL_{{\A},a}}=\sum_{j=1}^{2^n} U_j(t).
\end{equation}

\subsection{Homogeneous Strichartz estimates}
The homogeneous Strichartz estimates \eqref{Str-est} is a direct consequence of 
\begin{equation}\label{Str-est'}
\|U_j(t) u_0\|_{L^q_tL^p_x(\mathbb{R}\times \R^n)}\leq
C\|u_0\|_{L^2(\R^n)},\quad j=1,\cdots, 2^n.
\end{equation}
We will use the $L^2$-estimates and the localized dispersive estimates to prove \eqref{Str-est'}.
To this end, we need a variant of Keel-Tao's abstract argument.

\begin{proposition}\label{prop:semi-1}
Let $(X,\mathcal{M},\mu)$ be a $\sigma$-finite measure space and
$S: \mathbb{R}\rightarrow B(L^2(X,\mathcal{M},\mu))$ be a 
measurable map satisfying that, for some constants $\sigma>0$, $2\leq p_0\leq +\infty$, there exists a constant $C$ such that
\begin{equation}\label{md-1}
\begin{split}
\|S(t)\|_{L^2\rightarrow L^2}&\leq C,\quad t\in \mathbb{R},\\
\|S(t)S(s)^*f\|_{L^{p_0}}&\leq
C|t-s|^{-\sigma(1-\frac2{p_0})}\|f\|_{L^{p_0'}},\quad 2\leq p_0\leq +\infty.
\end{split}
\end{equation}
Then for every pair $q,p\in[2,\infty]$ such that $(q, p,\sigma)\neq
(2,\infty,1)$ and
\begin{equation*}
\frac{1}{q}+\frac{\sigma}{p}=\frac\sigma 2,\quad 2\leq p\leq p_0, \quad q\geq 2,
\end{equation*}
there exists a constant $\tilde{C}$ depending only on $C$, $\sigma$,
$q$ and $p$ such that
\begin{equation*}
\Big(\int_{\mathbb{R}}\|S(t) u_0\|_{L^p(X)}^q dt\Big)^{\frac1q}\leq \tilde{C}
\|u_0\|_{L^2(X)}.
\end{equation*}
\end{proposition}

\begin{proof} If $p_0=\infty$, this is precisely the Keel-Tao abstract Strichartz estimate from
\cite{KT}. One can repeat Keel-Tao's argument to prove this proposition with an additional restriction $p\leq p_0$ (since our condition is restricted to this range now).
\end{proof}

\begin{lemma}\label{lem:L2-est} Let $U_j(t)$ be defined in \eqref{def:Uj}, then 
\begin{equation}\label{est:L2}
\|U_j(t)u_0\|_{L^2(\R^n)}\lesssim \|u_0\|_{L^2(\R^n)}.
\end{equation}
\end{lemma}

\begin{proof} By the definition of $U_j(t)$ in \eqref{def:Uj} and $U(t)$ in \eqref{U_jsum}, since $Q_j$ is bounded from $L^2$ to $L^2$, it is easy to see
\begin{equation*}
\|U_j(t)u_0\|_{L^2(\R^n)}\lesssim \|U(t)u_0\|_{L^2(\R^n)}.
\end{equation*}
If $u_0\in L^2(\R^n)$, as \eqref{sep.v}, we expand it as $$u_0(x)=\sum_{k\in\N} c_k(r)\psi_k(\hat{x}).$$
Then by \eqref{funct-S} and \eqref{equ:knuk}, we obtain
\begin{equation*}
\begin{split}
 U(t)u_0&=\int_0^\infty \int_{\mathbb{S}^{n-1}} K(r_1,\hat{x},r_2,\hat{y}) u_0(r_2,\hat{y})\, r_2^{n-1} dr_2 d\hat{y}\\
 &=\sum_{k\in\N}\psi_k(\hat{x}) \mathcal{H}_{\nu_k}[e^{it\rho^2} (\mathcal{H}_{\nu_k} c_k)(\rho)],
 \end{split}
\end{equation*}
where $\mathcal{H}_{\nu}$ is the Hankel transform given in \eqref{hankel}. Therefore, we use the orthogonality of the eigenfunctions and the unitarity of the Hankel transform
$$\|\mathcal{H}_{\nu_k}f\|_{L^2(\rho^{n-1}d\rho)}=\|f\|_{L^2(r^{n-1}dr)}$$
to obtain \eqref{est:L2}.

\end{proof}

By the definition of \eqref{def:Uj} again,  we see that
\begin{equation*} 
U_j(t)U_j^*(s)=Q_j e^{i(t-s)\LL_{{\A},a}} Q_j 
\end{equation*}
due to that $Q_j^*=Q_j $. Hence we have the following decay estimates by Proposition
\ref{prop:LqLq'}
\begin{equation*}
\|U_j(t)U_j^*(s)f\|_{L^\infty}\lesssim |t-s|^{-n/2}\|f\|_{L^1}
\end{equation*}
provided $\alpha\geq 0$; while if $-(n-2)/2<\alpha<0$,
\begin{equation*}
\|U_j(t)U_j^*(s)f\|_{L^{p_0}}\lesssim |t-s|^{-\frac n2(1-\frac2{p_0})}\|f\|_{L^{p_0'}}, \quad p_0\in [2, p(\alpha)).
\end{equation*}
As a consequence of the Keel-Tao abstract Strichartz estimate in
Proposition \ref{prop:semi-1}, we obtain 
\begin{equation}\label{est:Uj}
\|U_j(t)u_0\|_{L^q(\mathbb{R}; L^p(\R^n))}\lesssim
\|u_0\|_{L^2(\R^n)},
\end{equation}
where $(q,p)\in\Lambda_0$ is sharp $\frac n2$-admissible, which is defined in \eqref{adm-p}. If $-(n-2)/2<\alpha<0$, we additionally require that
$2\leq p<p(\alpha)=-\frac n\alpha$, but when $(q,p)\in\Lambda_0$ one has $2\leq p\leq \frac{2n}{n-2}<-\frac n\alpha$ due to $\nu_0>0$. Therefore, we have proved 
\eqref{Str-est}.


\subsection{Inhomogeneous Strichartz estimates}\label{subsec:inhom}
In this subsection, we prove \eqref{eq:inhom}, except for the double endpoint $(q,r)= (\tilde q, \tilde r) = (2,\frac{2n}{n-2})$ for $n \geq 3$. Recall
$U(t)=e^{it\LL_{{\A},a}}: L^2\rightarrow L^2$. We have already established the homogenous Strichartz estimate \eqref{Str-est},
\begin{equation*}
\|U(t)u_0\|_{L^q_t(\R;L^p_x(\R^n))}\lesssim\|u_0\|_{L^2(\R^n)}
\end{equation*} which hold for all $(q,p)\in \Lambda_0$.
By a $TT^*$-type argument, the estimate is equivalent to
\begin{equation*}
\Big\|\int_{\R}{U}(t){U}^*(s)F(s)ds\Big\|_{L^q_t(\R;L^p_x(\R^n))}\lesssim\|F\|_{L^{\tilde{q}'}_t(\R;L^{\tilde{p}'}_x(\R^n))},
\end{equation*}
where both $(q,r)$ and $(\tilde{q},\tilde{r})$ belong to $\Lambda_0$.
By the Christ-Kiselev lemma \cite{CK}, we obtain for $q>\tilde{q}'$
\begin{equation}
\Big\|\int_{s<t}{U}(t){U}^*(s)F(s)ds\Big\|_{L^q_t(\R;L^p_x(\R^n))}\lesssim\|F\|_{L^{\tilde{q}'}_t(\R;L^{\tilde{p}'}_x(\R^n))}.
\end{equation}
Notice that $\tilde{q}'\leq 2 \leq q$, therefore we have proved \eqref{eq:inhom} except at the double endpoint
$(q,p)=(\tilde{q},\tilde{p})=(2,\frac{2n}{n-2})$. 

\section{The proof of Theorem \ref{thm:Strichartz'} } \label{sec: strichartz, Hs level}

In this section, we prove Theorem \ref{thm:Strichartz'} by using Theorem \ref{thm:Stri} and the Littlewood-Paley theory from Section \ref{sec:LP}, applied to the operator $\LL_{{\A},a}$.

\begin{proof}[The proof of Theorem \ref{thm:Strichartz'}.]  We begin by proving \eqref{Str-est-s} when $(q,p)$ satisfies \eqref{adm-p-s'}.
As noted in Remark \ref{rem:set}, we must have $s\in [0,1+\nu_0)$, otherwise  the set $\Lambda_{s,\nu_0}$ would be nonempty.
When $(q,p)\in \Lambda_{s,\nu_0}$, we have  $2\leq p<p(\alpha)$.
Thus, using \eqref{square}, \eqref{est:Bern} and Theorem \ref{thm:Stri}, we can estimate as follows. First, applying the Littlewood-Paley decomposition \eqref{square}, we obtain
\begin{equation}
\begin{split}
\|e^{it\LL_{{\A},a}}u_0\|^2_{L^q(\R;L^p(\R^n))}&\lesssim \sum_{j\in\Z}\big\|\varphi(2^{-j}\sqrt{\LL_{{\A},a}}) e^{it\LL_{{\A},a}}u_0\big\|^2_{L^q(\R;L^{p}(\R^n))}.
\end{split}
\end{equation}
Next, applying the Bernstein inequality \eqref{est:Bern}, the right hand side above is controlled by
\begin{equation}
\begin{split}
& \sum_{j\in\Z}2^{2nj(\frac1{\tilde{p}}-\frac1p)}\big\|e^{it\LL_{{\A},a}}\varphi(2^{-j}\sqrt{\LL_{{\A},a}}) u_0\big\|^2_{L^q(\R;L^{\tilde{p}}(\R^n))}.
\end{split}
\end{equation}
Now, using Theorem \ref{thm:Stri} and the definition of the Sobolev space from \eqref{Sobolev-n}, the quantity above is controlled by
\begin{equation}
\begin{split}
& \sum_{j\in\Z}2^{2js}\big\|\varphi(2^{-j}\sqrt{\LL_{{\A},a}}) u_0\big\|^2_{L^{2}(\R^n))}=\|u_0\big\|^2_{\dot H_{{\A},a}^{s}(\R^n))}.
\end{split}
\end{equation}
Here, we use
\begin{equation*}
s=n(1/\tilde{p}-1/p),\quad 2/q=n(1/2-1/\tilde{p}).
\end{equation*}
Therefore, we have proved \eqref{Str-est-s} for $(q,p)\in \Lambda_{s,\nu_0}$.\vspace{0.2cm}

To finish the proof of Theorem \ref{thm:Strichartz'}, we provide a counterexample demonstrating that \eqref{Str-est-s} fails when $p\geq p(\alpha)$.  Let $u_0(r)=(\mathcal{H}_{\nu_0}\chi)(r)$ be independent of $\hat{x}$, where $\chi\in\CC_c^\infty([1/2,1])$ has values in $[0,1]$
and $\mathcal{H}_{\nu_0}$ denotes the Hankel transform as defined in \eqref{hankel}. Due to the compact support of  $\chi$ and the
unitarity of $\mathcal{H}_{\nu_0}$ on $L^2$, we have $\|u_0\|_{\dot H^s}\leq C$. For this choice of $u_0$, we will show that 
\begin{equation*}
\begin{split} \|e^{it\LL_{{\A},a}}u_0\|_{L^q(\R;L^p(\R^n))}=\infty, \quad (q,p)\in\Lambda_s, \quad p\geq p(\alpha). \end{split}
\end{equation*}
Recalling the expression for $e^{it\LL_{{\A},a}}u_0$ from \eqref{funct-S}, we write 
\begin{equation*}
\begin{split} e^{it\LL_{{\A},a}}u_0&=\int_0^\infty(r\rho)^{-\frac{n-2}2}J_{\nu_0}(r\rho)e^{
it\rho^2}(\mathcal{H}_{\nu_0}u_0)(\rho)\rho^{n-1}d\rho,
\end{split}
\end{equation*}
which simplifies to 
\begin{equation*}
\begin{split} e^{it\LL_{{\A},a}}u_0=\int_0^\infty(r\rho)^{-\frac{n-2}2}J_{\nu_0}(r\rho)e^{
it\rho^2}\chi(\rho)\rho^{n-1}d\rho:=Z.
\end{split}
\end{equation*}
We aim to prove that:
  \begin{equation}\label{est:aim}
    \|Z\|_{L^q(\R;L^p(\R^n))} =\infty,
    \qquad p\geq p(\alpha).
  \end{equation}
  From the series expansion of $J_{\nu_0}(r)$ at $r=0$,
  we have
  \begin{equation}\label{Bessel1}
    J_{\nu_0}(r)=C_{\alpha}r^{\nu_0}+S_{\nu_0}(r)
  \end{equation}
  where
  \begin{equation}\label{Bessel3}
    |S_{\nu_0}(r)|\leq C_{\nu_0} r^{1-\nu_0},
    \qquad
    r\in(0,2].
  \end{equation}
  Then, for any $0<\epsilon<1$, we  estimate $Z$ as follows:
  \begin{equation*}
  \begin{split}
    \|Z\|_{L^q(\R;L^p(\R^n))}&\ge
    \|Z\|_{L^q_{t}([0,1/2];L^p_{r^{n-1}dr}[\epsilon,1])}\\
    &\ge
    C_{\alpha}\|P\|_{L^q_{t}([0,1/4];L^p_{r^{n-1}dr}[\epsilon,1])}
    -
    \|Q\|_{L^q_{t}([0,1/2];L^p_{r^{n-1}dr}[\epsilon,1])},
    \end{split}
  \end{equation*}
  where $\alpha=\nu_0-(n-2)/2$ and
  \begin{equation*}
    P= \int_{0}^{\infty}(r \rho)^{\alpha}e^{it \rho^2}\chi(\rho)
      \rho^{n-1} d \rho,
    \qquad
    Q=\int_{0}^{\infty}(\rho r)^{-\frac{n-2}2}S_{\nu_0}(\rho)e^{it \rho^2}\chi(\rho)
      \rho^{n-1} d \rho.
  \end{equation*}
  Now, on one hand by \eqref{Bessel3} and the fact that $2\leq p\leq \frac{2n}{n-2}$, we have
  \begin{equation*}
  \begin{split}
    \|Q\|_{L^q_{t}([0,1/2];L^p_{r^{n-1}dr}[\epsilon,1])}
   & \lesssim
    \left\|
      \int_{0}^{\infty}(r \rho)^{1+\alpha}\chi(\rho)\rho^{n-1} d \rho
    \right\|_{L^q_{t}([0,1/2];L^p_{r^{n-1}dr}[\epsilon,1])}\\
   & \lesssim \max\big\{\epsilon^{\alpha+1+\frac np},1\big\}\lesssim 1.
    \end{split}
  \end{equation*}
  On the other hand, we have
  \begin{equation*}
    \|P\|_{L^q([0,1/4];L^p_{r^{n-1}dr}[\epsilon,1])}=
    \left(\int_0^{\frac14}\left( \int_{\epsilon}^1 \left|\int_0^\infty (r\rho)^{\alpha}e^{
    it\rho^2}\chi(\rho)\rho^{n-1} d\rho\right|^p r^{n-1} dr\right)^{q/p}dt\right)^{1/q}
  \end{equation*}
  and by the assumption $p\geq p(\alpha)=-\frac n\alpha$
  \begin{equation*}
  \begin{split}
    &\gtrsim
    \left(\int_0^{\frac14}
    \left|\int_0^\infty \rho^{\alpha}
    e^{it\rho^2}\chi(\rho)\rho^{n-1} d\rho\right|^{p}dt\right)^{1/p}
    \times 
    \begin{cases}
      \epsilon^{\alpha+\frac np} &
      \text{if}\quad n+p\alpha<0\\
      \ln\epsilon &
      \text{if}\quad p\alpha+n=0
    \end{cases}\\
   & \gtrsim C
    \begin{cases}
      \epsilon^{\alpha+\frac np} &
     \text{if}\quad p>p(\alpha),\\
      \ln\epsilon &
        \text{if}\quad p=p(\alpha).    \end{cases}
    \end{split}
  \end{equation*}
  In the last inequality, we have used the fact that 
  $\cos(\rho^2 t)\geq 1/100$ for $t\in [0, 1/4]$ and 
  $\rho\in [1,2]$, so that
  \begin{equation}
  \begin{split}
  \left|\int_0^\infty \rho^{\alpha}e^{
  it\rho^2}\chi(\rho)\rho^{n-1} d\rho\right|\geq \frac1{100}\int_0^\infty \rho^{\alpha}\chi(\rho)\rho^{n-1} d\rho\geq c.
  \end{split}
  \end{equation}
  We now conclude the following:
  
  $\bullet$ If $p>p(\alpha)=-\frac n\alpha$, we obtain 
  \begin{equation*}
     \|Z\|_{L^q(\R;L^p(\R^n))} \geq
      c\epsilon^{\alpha+\frac np} -C 
      \to +\infty \qquad \text{as}\quad \epsilon\to 0;
  \end{equation*}
  
  $\bullet$ If $p=p(\alpha)$ we have
  \begin{equation*}
        \|Z\|_{L^q(\R;L^p(\R^n))} \geq   c\ln\epsilon-C\to +\infty \qquad \text{as}\quad \epsilon\to 0.
  \end{equation*}
This proves that $\|Z\|_{L^q(\R;L^p(\R^n))}=+\infty$ for $p\geq p(\alpha)$, which implies  \eqref{est:aim}. Thus, we have shown that \eqref{Str-est-s} fails when $p\geq p(\alpha)$, completing the proof of Theorem \ref{thm:Strichartz'}.

\end{proof}

\begin{center}

\end{center}

\end{document}